\documentclass[a4paper,11pt,reqno,twoside]{amsart}

\usepackage{lmodern}
\usepackage[x11names]{xcolor}
\usepackage[english]{babel}
\usepackage[shortlabels]{enumitem}

\usepackage{latexsym}

\usepackage[x11names]{xcolor}
\usepackage[utf8]{inputenc}
\usepackage[T1]{fontenc}
\usepackage{csquotes}

\usepackage{amsmath}
\usepackage{amsfonts}
\usepackage{amsbsy}
\usepackage{amssymb}
\usepackage{amsthm}
\usepackage{amscd}

\usepackage{subcaption}
\usepackage{colortbl}

\usepackage{bookmark} 

\usepackage{marvosym} \usepackage{mathtools} \usepackage{stmaryrd}

\usepackage{tabularx}
\usepackage[font=footnotesize]{caption}
\usepackage{colortbl}

\usepackage[notref, notcite, color, final]{showkeys}
\definecolor{refkey}{rgb}{0.0,0.0,1.0}
\definecolor{labelkey}{rgb}{0.0,0.0,1.0}

\theoremstyle{plain}
\newtheorem{theorem}{Theorem}[section]
\newtheorem{proposition}[theorem]{Proposition}
\newtheorem{lemma}[theorem]{Lemma}
\newtheorem{corollary}[theorem]{Corollary}

\theoremstyle{definition}

\newtheorem{assumption}[theorem]{Assumption}
\newtheorem{example}[theorem]{Example}

\theoremstyle{remark}
\newtheorem{remark}[theorem]{Remark}

\usepackage{enumitem}

\numberwithin{equation}{section}

\allowdisplaybreaks

\newcommand{\dd}{\mathrm{d}}
\newcommand{\ii}{\mathrm{i}}
\newcommand{\ee}{\mathrm{e}}

\newcommand{\Bc}{\mathcal B}

\newcommand{\Fc}{\mathcal F}

\newcommand{\Ic}{\mathcal I}

\newcommand{\Lc}{\mathcal L}
\providecommand{\Mc}{}
\renewcommand{\Mc}{\mathcal M}
\newcommand{\Nc}{\mathcal N}

\newcommand{\Sc}{\mathcal S}

\newcommand{\Xc}{\mathcal X}

\newcommand{\Zc}{\mathcal Z}

\newcommand{\D}{\mathbb D}

\newcommand{\N}{\mathbb N}

\newcommand{\R}{\mathbb R}

\newcommand{\Z}{\mathbb Z}

\newcommand{\Eb}{\mathbf{E}}

\newcommand{\Pb}{\mathbf{P}}

\newcommand{\Xb}{\mathbf{X}}

\usepackage{mathrsfs}

\newcommand{\Xs}{\mathscr X}

\newcommand{\Zs}{\mathscr Z}

\newcommand{\mf}{\mathfrak m}

\usepackage{dsfont}
\newcommand{\1}{\mathbf 1}

\renewcommand{\epsilon}{\varepsilon}
\renewcommand{\phi}{\varphi}

\newcommand{\argmin}{\operatornamewithlimits{argmin}}

\newcommand{\KL}{\mathrm{KL}}
\newcommand{\TV}{\mathrm{TV}}

\newcommand{\Cov}{\operatorname{Cov}}

\newcommand{\defeq}{\vcentcolon=}
\newcommand{\eqdef}{=\vcentcolon}

 \newcommand{\trace}{\mathrm{tr}}

\usepackage{mathtools} \renewcommand{\complement}{\mathsf{c}}

\newcommand{\Xbar}{\bar X}
\newcommand{\xtilde}{\widetilde x}
\newcommand{\xtildeprime}{\xtilde^\prime}
\newcommand{\Xtilde}{\widetilde X}
\newcommand{\Xtildebar}{\bar{\Xtilde}}
\newcommand{\xibar}{\overline \xi}

\newcommand{\e}{\boldsymbol{\mathrm{e}}}
\newcommand{\knast}{k_n^{\ast}}
\newcommand{\absj}{\lvert j\rvert}

\newcommand{\Zbar}{\bar Z}
\newcommand{\Zprimebarn}{\overline{Z}^\prime_n}
\newcommand{\Xbarn}{\bar X_n}
\newcommand{\xibarn}{\bar \xi_n}
\newcommand{\Ihat}{\widehat I}
\newcommand{\ahat}{\widehat a}

\newcommand{\fhat}{\widehat f}

\newcommand{\pen}{\operatorname{pen}}
\newcommand{\pentilde}{\widetilde \pen}
\newcommand{\mfhat}{\widehat \mf}
\newcommand{\mfprime}{\mf^\prime}

\newcommand{\ftilde}{\widetilde f}

\newcommand{\Acomplement}{A^\complement}

\newcommand{\privpar}{\alpha}

\newcommand{\Psinpp}{\Psi_{n,\privpar}}
\newcommand{\Laplace}{\Lc}
\newcommand{\ftheta}{f^\theta}
\newcommand{\spn}{\operatorname{span}}

\newcommand{\rbar}{\bar r}
\newcommand{\mfast}{\mf^\ast}

\newcommand{\Itilde}{\widetilde I}
\newcommand{\Gtilde}{\widetilde G}

\newcommand{\Xitilde}{\widetilde \Xi}

\newcommand{\Sbar}{\overline S}
\newcommand{\vhat}{\widehat v}
\newcommand{\kappatilde}{\widetilde \kappa}

\usepackage[paper=a4paper,left=30mm,right=30mm,top=30mm,bottom=30mm]{geometry}

\usepackage{hyperref}
\hypersetup{colorlinks=true,urlcolor=RoyalBlue4,linkcolor=RoyalBlue4,citecolor=RoyalBlue4,naturalnames=true,hypertexnames=false}

\usepackage[backend=biber, giveninits, style=alphabetic, uniquelist=false, natbib=true, doi=false, url=false, isbn = false]{biblatex}
\DeclareFieldFormat[article]{title}{#1} \DeclareFieldFormat[book]{title}{\itshape #1}
\DeclareFieldFormat[thesis]{title}{#1}
\DeclareFieldFormat[incollection]{title}{\itshape #1}
\DeclareFieldFormat[inproceedings]{title}{\itshape #1}
\DeclareFieldFormat[article]{author}{\color{blue} #1}
\DeclareFieldFormat[article]{pages}{#1.}
\DeclareFieldFormat[incollection]{pages}{#1.} 

\renewbibmacro{in:}{\ifentrytype{article}{}{\printtext{\bibstring{in}\intitlepunct}}}

\DeclareNameAlias{sortname}{last-first}
\DeclareNameAlias{default}{last-first}

\renewbibmacro*{volume+number+eid}{\printfield{volume}\setunit*{\addnbthinspace}\printfield{number}\setunit{\addcomma\space}\printfield{eid}}
\DeclareFieldFormat[article]{number}{\mkbibparens{#1}}

\renewbibmacro*{publisher+location+date}{\printlist{publisher}\iflistundef{location}
	{\setunit*{\addcomma\space}}
	{\setunit*{\addcomma\space}}\printlist{location}\setunit*{\addcomma\space}\usebibmacro{date}\newunit}

 \DeclareFieldFormat[article]{number}{}

\title[]{Adaptive spectral density estimation by model selection under local differential privacy}
\author{Martin Kroll}
\address{Ruhr-Universität Bochum\\Fakultät für Mathematik\\D-44780 Bochum}
\email{martin.kroll-k9x@rub.de}

\date{\today}
\subjclass[2010]{62G05 (primary), and 62M10 (secondary)} 
\keywords{Differential privacy, spectral density estimation, orthonormal series estimator, adaptive estimation, model selection.}

\begin{document}

\begin{abstract}
We study spectral density estimation under local differential privacy. A\-no\-ny\-mi\-zation is achieved through truncation followed by Laplace perturbation. We select our estimator from a set of candidate estimators by a penalized contrast criterion. This estimator is shown to attain nearly the same rate of convergence as the best estimator from the candidate set. A key ingredient of the proof are recent results on concentration of quadratic forms in terms of sub-exponential random variables obtained in \cite{goetze2019concentration}. We illustrate our findings in a small simulation study.
\end{abstract} 
\maketitle

\section{Introduction}

Let $X=(X_t)_{t \in \Z}$ be a stationary time series with autocorrelation function $\gamma \colon \Z \to \R$ defined via $\gamma_k = \gamma(k) = \Cov(X_t, X_{t+k})$ for any $t \in \Z$.
Then, under the assumption that the series $(\gamma_k)_{k \in \Z}$ is absolutely convergent, its spectral density $f$ can be represented by the inversion formula as
\begin{equation}\label{eq:spec:dens}
f(\omega) = \frac{1}{2\pi} \sum_{k \in \Z} \gamma_k e^{-\ii \omega k}, \qquad \omega \in [-\pi,\pi].
\end{equation}

The present paper treats the nonparametric estimation of $f$ from a finite snippet $X_{1:n} = (X_1,\ldots,X_n)$ of length $n$ under privacy constraints.
More precisely, we assume that each $X_i$, $i = 1,\ldots,n$ belongs to another data holder who is willing to publish an anonymized version $Z_i$ of the actual value $X_i$ only.
Then, the complete snippet $X_{1:n}$ is not accessible to the statistician, and estimation can be done based on the privatized snippet $Z_{1:n} = (Z_1,\ldots,Z_n)$ solely.
Such a situation might, for instance, be of relevance for so-called random walk survey designs \cite{bennett1991simplified,lemeshow1985surveys} where people are successively interviewed and the next person to be interviewed is determined by a precisely defined random walk strategy.
Hence, one might suppose that data of consecutively interviewed persons should incorporate some amount of dependency since they live nearby each other.
Since survey interviews often aim at collecting data concerning sensitive social and biological data concerning health \cite{flynn2013modified} there is certainly need for anonymization.

As our mathematical setup for privacy we use the notion of local $\privpar$-differential privacy that has gained increasing popularity in the statistics community in recent years.
Until now, theoretical research in this framework has focused on models with independent observations and estimation tasks like density estimation \cite{duchi2018minimax,butucea2020local}, estimation of functionals \cite{rohde2020geometrizing,butucea2020interactive}, testing \cite{berrett2020locally,lam-weil2020minimax,butucea2020interactive}, and classification \cite{berrett2019classification}.
To the best of our knowledge, this work is the first one that applies the concept of differential privacy to time series data and the task of estimating the dependency structure of a process under privacy restrictions. 

Of course, in the classical scenario without any privacy restriction there exists an overwhelming amount of literature on spectral density estimation from a snippet of finite length in both parametric and nonparametric models \cite{comte2001adaptive,dahlhaus1989efficient,davies1973asymptotic,efromovich1998data,fox1986large,golubev1993nonparametric,neumann1996spectral,soulier2000estimation,taniguchi1987minimum}.
From a methodological point of view we explicitly point out the paper \cite{comte2001adaptive} that is our point of origin and uses the penalized contrast approach (that we will also use later in Section~\ref{s:adaptation}) in the non-private framework.
For this reason, we give in the following a recap of this technique.

\subsection*{The non-private model selection device from \cite{comte2001adaptive} in a nutshell}

The so-called periodogram is the point of origin of many procedures for spectral density estimation but it is has to be smoothed in order to obtain consistency or even rate optimal estimators. 
The \emph{centred periodogram}, based on $n$ consecutive observations of the time series, is defined via  
\begin{equation}\label{eq:def:I_n^X}
	I_n^X(\omega) = \frac{1}{2\pi n} \left\lvert \sum_{t=1}^n (X_t - \Xbar_n) e^{-\ii t \omega} \right\rvert^2,
\end{equation}
where $\Xbar_n$ is the sample mean of the observed snippet $X_{1:n}$.

One possibility to obtain a smoothed version of the periodogram is projection of $I_n^X$ to a finite-dimensional subspace $S_\mf$ of $L^2([-\pi,\pi])$, say $S_\mf = \spn(\phi_i)_{i \in \Ic_\mf}$, where $(\phi_i)_{i \in \Ic_\mf}$ is an orthonormal basis of $L^2([-\pi,\pi])$ and $\Ic_\mf$ a finite subset.
We denote $D_{\mf} = \dim(S_\mf) = \lvert \Ic_\mf \rvert$.
The choice of a subspace $S_\mf$ might be interpreted as the choice of a finite-dimensional model $\mf$ (which explains the choice of the letter $\mf$ here).
Potential models include spaces generated by trigonometric functions, regular piecewise polynomials, general piecewise polynomials, regular compactly supported periodic wavelets, and general compactly supported periodic wavelets; see Section~2.2 in \cite{comte2001adaptive} for a detailed description of all of these models.
With such a model we associate the estimator
\begin{equation*}
	\fhat_{\mf}^X = \sum_{i \in \Ic_\mf} \widehat a_i^X \phi_i
\end{equation*}
with coefficients
\begin{equation*}
	\ahat_i^X = \int_{-\pi}^{\pi} I_n^X(\omega) \phi_i(\omega)\dd \omega.
\end{equation*}
Note that so defined estimator minimizes the quantity
\begin{equation*}
	\int_{-\pi}^{\pi} (I_n(\omega) - t(\omega))^2 \dd \omega, 
\end{equation*}
or, equivalently, the contrast $\Upsilon_n(t) = \int_{-\pi}^{\pi} t^2(\omega)\dd \omega - 2 \int_{-\pi}^{\pi} t(\omega) I_n^X(\omega) \dd \omega$, over all $t \in S_\mf$.

An upper risk bound for the estimator $\fhat^X_\mf$ can be derived from the following decomposition (defining $f_\mf$ as the projection of $f$ onto the space $S_\mf$):
\begin{align}
	\Eb  \lVert f - \fhat_\mf^X \rVert^2 &= \lVert f - f_\mf \rVert^2 + \Eb \lVert \fhat_\mf^X - f_\mf \rVert^2 \notag\\
	&\leq \lVert f - f_\mf \rVert^2 + C(f) \cdot \frac{D_ \mf}{n},\label{eq:upper:nonprivate}
\end{align}
where the inequality is taken from Equation~(5) in \cite{comte2001adaptive}.
It is based on the following assumption (Assumption~2 in \cite{comte2001adaptive}) that we will adopt for this work.
\begin{assumption}\label{ASS:GAMMA_K}
	The autocovariance function $\gamma$ of the time series $X$ is such that $\sum_{k \in \Z} \lvert \gamma_k \rvert = M < + \infty$ and $\sum_{k \in \Z} \lvert k\gamma_k^2 \rvert = M_1 < +\infty$.
\end{assumption}

From the results in \cite{comte2001adaptive} it is easy to see that the optimal rate that can be achieved for smooth spectral density functions belonging to a Sobolev space with smoothness parameter $s$ is $n^{-2s/(2s+1)}$.
However, as often in nonparametric statistics, the optimal model from a set of potential models that has to be selected to reach this rate can be chosen directly only if a priori knowledge concerning the smoothness is available.
Since such knowledge is usually not given, one has to find a method for model selection that is completely data-driven.
A by now classical method for this purpose is model selection \cite{birge1997from,barron1999risk,massart2007concentration}.
This general toolbox has been used by F. Comte in \cite{comte2001adaptive} for adaptive spectral density estimation in the non-private case where $X_{1:n}$ is observable. Her method consists in choosing a model as the minimizer $\mfhat$ of a penalized contrast criterion over a set $\Mc_n$ of potential models, that is,
\begin{equation*}
	\mfhat = \argmin_{\mf \in \Mc_n} \Upsilon_n(\fhat_\mf^X) + \pen(\mf).
\end{equation*}
Here $\Upsilon_n$ is a contrast function (for instance, the one defined above) and $\pen \colon \Mc_n \to [0,\infty)$ a penalty function that penalizes too complex potential models.
Usually, $\pen$ is a monotone function in the dimension $D_\mf$ of the space $S_\mf$.
In \cite{comte2001adaptive} it has been shown that, under sufficiently mild assumptions, the estimator $\fhat_{\mfhat}^X$ behaves nearly as well as the oracle given by the optimal model from the collection:
\begin{equation}\label{eq:oracle:comte}
	\Eb \lVert \fhat_{\mfhat}^X - f \rVert^2 \leq C_1 \inf_{\mf \in \Mc_n} \{ \lVert f - f_{\mf} \rVert^2 + \pen(\mf) \} + \frac{C_2}{n}. 
\end{equation}
Here, the constant $C_1$ is purely numerical whereas $C_2$ might depend on $f$ through its sup-norm, and additionally on quantities related to the class $\Mc_n$ of potential models (but, of course, not on $n$).
If the penalty term can be chosen of the same order as the variance term $D_\mf/n$ in \eqref{eq:upper:nonprivate} (maybe up to logarithmic factors), then the adaptive estimator attains the same rate as the best possible estimator over all potential models (up to logarithmic factors).

\subsection*{Contributions of the paper}

The principal purpose of this work is to derive an oracle inequality in the spirit of \eqref{eq:oracle:comte} when only anonymized data are available.
The main difficulty in this scenario is that the periodogram $I_n^X$ is not directly available.
Hence, one approach would be to define differentially private $Z_i$ in a way such that a suitable substitute $I_n^Z$ for $I_n^X$ can be defined in terms of the $Z_i$ only.
We introduce a procedure to define such $Z_{1:n}$ in the framework of $\privpar$-differential privacy by a combination of truncation and Laplace perturbation.
Using the privatized version of the periodogram, one can then apply the general toolbox as in the non-private case.
We first consider upper bounds in the spirit of \eqref{eq:upper:nonprivate} for projections of $I_n^Z$ to finite-dimensional spaces $S_\mf$ for fixed models $\mf$.
In the specific case where the privacy level $\privpar$ is fixed and interpreted as a constant whereas $n$ tends to $+\infty$, the rate of convergence over Sobolev ellipsoids that we obtain is the same as in the non-private setup up to logarithmic factors. 
Complementary to these upper bounds, we also state some first lower bound results that show that in some cases there might be a loss in the rate caused by the privacy level $\privpar$ when it is allowed to vary with $n$.
The main theoretical result of this paper is an oracle inequality in the spirit of \eqref{eq:oracle:comte} for private data.
For our completely data-driven estimator $\ftilde = \fhat_{\mfhat}^Z$ with the model $\mfhat$ determined via a model selection device (with a penalty that is adapted to the privacy framework), we derive
\begin{equation*}
	\Eb \lVert \ftilde - f \rVert^2 \leq C_1 \inf_{\mf \in \Mc_n} [ \lVert f - f_{\mf} \rVert^2 + \pen(\mf) ] + C_2 \max \left\lbrace \frac{1}{n}, \frac{\log^2(n)}{n^3 \privpar^4} \right\rbrace
\end{equation*}
where $\privpar$ is the privacy parameter (see Section~\ref{s:privacy} for the significance of this parameter).
In contrast to \eqref{eq:oracle:comte}, also the penalty depends on the privacy parameter $\privpar$ in our case.
However, as in the non-private setup, the adaptive estimator suffers at most from an additional loss in extra logarithmic terms in contrast to the optimal possible estimator taken from the considered collection of models.
Remarkably, recent results on the concentration of quadratic forms in sub-exponential random variables \cite{goetze2019concentration} turn out to be useful for our theoretical analysis.
From a methodological point of view the present work might be of interest since it is, to the best of the author's knowledge, the first paper where model selection has been used to perform adaptive estimation under privacy constraints (\cite{butucea2020local} uses wavelet estimators to achieve adaptation in the privacy setup).

\subsection*{Notation}

For real numbers $a,b$ we set $\llbracket a,b\rrbracket = [a,b] \cap \Z$.
By $\Laplace(b)$ we denote the Laplace distribution with parameter $b$, by $\Nc(\mu,\Sigma)$ the normal distribution with mean $\mu$ and covariance matrix $\Sigma$.
With $P_H v$ we denote the projection of a vector $v$ to some subspace $H$.
With $E_n$ we denote the $n \times n$-identity matrix, with $\boldsymbol 0_n$ the $n \times n$-zero matrix, and with $\vec c$ the $n \times 1$-vector containing only the value $c \in \R$ as entry.
$\rho(A)$ denotes the spectral radius of a matrix $A$.

For any real-valued random variable $X$ and $\beta > 0$ define the (quasi-)norm
\begin{equation*}
	\lVert X \rVert_{\psi_\beta} \defeq \inf \left\lbrace t > 0 : \Eb \exp \left( \frac{\lvert X \rvert^\beta}{t^\beta} \right) \leq 2 \right\rbrace 
\end{equation*}
(as usual, one puts $\inf \emptyset = +\infty$).
The (quasi-)norms $\lVert \cdot \rVert_{\psi_\beta}$ are called \emph{exponential Orlicz norms}.
By $\lVert \cdot \rVert$ we denote the usual $L^2$-norm, by $\lVert \cdot \rVert_{\mathrm{op}}$ the operator norm of a matrix.

We write $a_n \lesssim b_n$ if $a_n \leq C b_n$ for some purely numerical constant $C$ and all sufficiently large $n$.
Throughout the paper, $C$ denotes a generic constant whose value might change with every appearance.
By writing $C(\ldots)$ we indicate the dependence of a numerical constant on one or several parameters that are listed within the brackets.

\subsection*{Organization of the paper}
The paper is organized as follows. Section~\ref{s:privacy} introduces the notion of $\privpar$-differential privacy and we introduce our algorithm to anonymize time series data.
Section~\ref{s:minimax} is devoted to the derivation of upper risk bounds for fixed models $\mf$, and we also give some lower bound results.
In the main Section~\ref{s:adaptation} we state the oracle inequality for privatized time series data.
A small sample simulation study is presented in Section~\ref{s:sim} followed by a summary in Section~\ref{s:discussion} where we also indicate directions for further research. 
\section{Privacy}\label{s:privacy}

\subsection*{The notion of local $\privpar$-differential privacy}
Let us denote by $X_1,\ldots,X_n$ the un\-anony\-mized random variables, that is $X_{1:n} = (X_1,\ldots,X_n)$ is a snippet from the stationary time series $X$ whose spectral density is the quantity of interest.
We assume that each $X_i$ belongs to a certain data holder who does not want to publish the value $X_i$ but only an anonymized version of it, which will be denoted with $Z_i$.
A theoretical framework for formalizing the vague catchwords \emph{anonymization} and \emph{privacy} is $\privpar$-differential privacy which originally goes back to \cite{dwork2006automata} and has obtained increasing interest in the statistics community within the last decade.
There is a distinction between \emph{global} differential privacy (for instance, considered in \cite{hall2013differential}) where a trusted curator is given access to the complete data (that is, in our case, the snippet $X_{1:n}$) and a privatized version of standard estimators can be published, and \emph{local} differential privacy where the original data are anonymized directly by the data holders and estimation has to be performed using the resulting private snippet $Z_{1:n}$.
We will stick to this latter framework of local differential privacy in this paper.
Under local privacy, data are successively obtained applying appropriate Markov kernels.
More precisely, given $X_i = x_i$ and $Z_j = z_j$ for $j = 1,\ldots,i-1$, the $i$-th privatized output $Z_i$ is drawn as
\begin{equation}\label{EQ:DEF:Z_i:GEN}
	Z_i \sim Q_i (\cdot \mid X_i = x_i, Z_1=z_1,\ldots,Z_{i-1} = z_{i-1})
\end{equation}
for Markov kernels $Q_i: \Zs \times (\Xc \times \Zc^{i-1}) \to [0,1]$ with $(\Xc,\Xs)$, $(\Zc,\Zs)$ denoting the measure spaces of non-private and private data, respectively (cf.~Figure~2 in \cite{duchi2018minimax} for a representation of this sampling scheme as a graphical model).
In this paper, we propose a \emph{non-interactive} algorithm where the random value $Z_i$ depends on $X_i$ only: thus, there is no dependence on previously generated $Z_i$'s on the right-hand side of Equation~\eqref{EQ:DEF:Z_i:GEN}.
We also dispense with the dependence of $Q_i$ on $i$, that is, we consider procedures with
\begin{equation*}
	Z_i \sim Q(\cdot \, | \, X_i = x_i)
\end{equation*}
for all $i\in \llbracket 1,\ldots,n \rrbracket$.

\smallskip

The quantification of privacy is achieved via the notion of $\privpar$-differential privacy.
In our context, this notion means that the estimate
\begin{equation}\label{eq:def:privpar:privacy}
	\sup_{A \in \Zs} \frac{Q_i(A \, | \, X_i = x)}{Q_i(A \, | \, X_i = x')} \leq \exp(\privpar)
\end{equation}
is supposed to hold for all $x,x' \in \Xc$.
If there exist densities $q(z \, | \, X = x)$ for the Markov kernel for all $x \in \Xc$ it is easy to verify that condition \eqref{eq:def:privpar:privacy} is equivalent to
\begin{equation}\label{eq:def:privpar:privacy:dens}
	\sup_{z \in \Zc} \frac{q(z \, | \, X_i = x)}{q(z \, | \, X_i = x')} \leq \exp(\privpar)
\end{equation}
for all $x,x' \in \Xc$.

\subsection*{Anonymization procedure}
It well-known that adding centred Laplace distributed noise on bounded random variables with sufficiently large variance guarantees $\privpar$-differential privacy \cite{duchi2018minimax,rohde2020geometrizing}.
We will use this general technique but have to transform the $X_i$ in a first step because we do not want to impose a boundedness assumption on the $X_i$ in general since this is obviously not satisfied in the most important case of Gaussian time series.
This transformation consists in a truncation of the original $X_i$.
More precisely, we put
\begin{equation}\label{eq:def:Xtilde_i}
	\Xtilde_i = (X_i \wedge \tau_n) \vee (-\tau_n), \qquad i \in \llbracket 1,n\rrbracket
\end{equation}
where $\tau_n > 0$.
Note that the truncation can be performed locally by the data holders once all of them have agreed on the value $\tau_n$.
Our estimators are quite sensitive to the choice of the threshold $\tau_n$.
On the one hand, we want $\tau_n$ to tend to $+\infty$ in order to bound the probability that truncation occurs for at least one variable $X_i$ by the rate of convergence that we aim at.
On the other hand, $\tau_n$ arises in the rates of convergence and should be as small as possible.
For our purposes, a logarithmically increasing (in terms of the snippet length $n$) sequence $\tau_n$ will turn out to be convenient.

By construction, we trivially have $\Xtilde_i \in [-\tau_n, \tau_n]$, and the above mentioned Laplace technique can be applied on the transformed data.

\begin{lemma} The random variables
	\begin{equation}\label{eq:def:Z_i}
	Z_i = \Xtilde_i + \xi_i
	\end{equation}
	with $\xi_i \text{ i.i.d.} \sim \Laplace(2\tau_n/\privpar)$ are $\privpar$-differentially private views of the original $X_i$.
\end{lemma}

\begin{proof}
	We only have to check \eqref{eq:def:privpar:privacy:dens}.
	Recall that the density of a centred Laplace distributed random variable with scale parameter $b>0$ is given by $1/(2b) \exp \left( - \lvert x \rvert/b \right)$. 
	Put $\xtilde = (x \wedge \tau_n) \vee (-\tau_n)$ and $\xtildeprime = (x' \wedge \tau_n) \vee (-\tau_n)$.
	By the reverse triangle inequality, we have
	\begin{align*}
		\sup_{z \in \Zc} \frac{q(z \, | \, X = x )}{q(z \, | \, X = x^\prime )} &= \sup_{z \in \Zc} \exp \left( -\privpar \cdot \frac{\lvert z - \xtilde \rvert}{2\tau_n} + \privpar \cdot \frac{\lvert z - \xtildeprime \rvert}{2 \tau_n} \right) \\
		&\leq \exp \left( \privpar \cdot \frac{\lvert \xtilde - \xtildeprime \rvert}{2\tau_n} \right)\\
		&\leq \exp(\privpar),
	\end{align*}
	and \eqref{eq:def:privpar:privacy:dens} holds.
\end{proof}

\begin{remark}
	Let use mention that the privacy mechanism defining the $Z_i$ in \eqref{eq:def:Z_i} is convenient for our purposes in this paper but not optimal in other scenarios.
	For instance, imagine that the $X_i$ are i.i.d. and the statistician wants to estimate the underlying probability density function $f$.
	Then, apart from the additional threshold, \eqref{eq:def:Z_i} defines a convolution model with Laplace distributed error density.
	Convolution models are well studied and it is known that the rate of convergence for $s$-smooth functions based on observations $Z_i$ is at least $n^{-2s/(2s+3)}$ \cite{fan1991optimal}.
	However, the optimal rate under local differential privacy (considering $\privpar$ as a fixed constant), that can only be attained using privacy mechanisms different from \eqref{eq:def:Z_i}), is $n^{-s/(s+1)}$ as has been shown in \cite{duchi2018minimax} and \cite{butucea2020local}.
	This emphasizes the fact that the privacy mechanisms to be used should not only depend on the available data but also on the statistical problem at hand.
\end{remark} 
\section{Risk bounds for fixed models}\label{s:minimax}

In this section, we propose an estimator of the spectral density $f$ based only on observations of the privatized data $Z_i$ as defined in \eqref{eq:def:Z_i}.
In this case, we derive an upper risk bound similar to \eqref{eq:upper:nonprivate} for any fixed model $\mf$.
As a consequence we obtain that, regarding the privacy parameter $\privpar$ as a fixed numerical constant, the proposed estimator attains the nearly same rate of convergence in terms of the snippet length $n$ as in the non-private setup up to an additional logarithmic factor.
Our estimator is based on the \emph{privatized periodogram}
\begin{equation*}
	I_n^Z(\omega) = \frac{1}{2\pi n} \left\lvert \sum_{t=1}^{n} (Z_t - \Zbar_n) e^{-\ii t \omega} \right\rvert^2.
\end{equation*}
The function $I_n^Z$ formally resembles the definition of the periodogram in \eqref{eq:def:I_n^X} with $X_i$ being replaced with $Z_i$.
Put $Z'_i = X_i + \xi_i$.
Then $Z_i' = Z_i$ holds whenever $X_i = \Xtilde_i$, that is, the value $X_i$ is not modified in the truncation step \eqref{eq:def:Xtilde_i}.
It is intuitively clear that in this 'nice' case one can hope to extract much more information from the dependency structure of the time series than in the case where truncation leads to a value $\Xtilde_i$ different from $X_i$.
This 'nice' event is formalized in the proofs of Theorems~\ref{thm:upper:minimax} and \ref{thm:adaptation} below via the event $A =  \{  X_i = \Xtilde_i \text{ for all } i \in \llbracket 1,n\rrbracket \}$.
For $i,j \in \llbracket 1,n\rrbracket$, the covariance between $Z'_i$ and $Z'_j$ can be calculated as
\begin{align*}
\Cov (Z'_i, Z'_j) &= \Cov \left( X_i + \xi_i, X_j + \xi_j \right)\\
&= \Cov(X_i,X_j) + \Cov(X_i,\xi_j) + \Cov(\xi_i,X_j) + \Cov(\xi_i, \xi_j)\\
&= \gamma_{\lvert i-j \rvert}  + \frac{8\tau_n^2}{\privpar^2} \, \delta_{ij},
\end{align*}
where $\delta_{ij}$ is the Kronecker delta.
Thus, by the inversion formula \eqref{eq:spec:dens}, we have
\begin{equation}\label{eq:fZprime}
	f^{Z'}(\omega) = f(\omega) + \frac{8\tau_n^2}{\privpar^2}
\end{equation}
with $f^{Z'}$ denoting the spectral density of the stationary time series $(Z'_t)_{t \in \Z}$.
There is only hope to be able to estimate this quantity if we can observe the $Z_i'$ for a significant amount of $i$.
This is the more likely the larger the threshold $\tau_n$ is chosen.
Under our technical assumptions that will be introduced below, a logarithmically increasing sequence $\tau_n$ guarantees that $Z_i = Z_i'$ for all $i \in \llbracket 1,n \rrbracket$ with sufficiently high probability.
In this scenario, it then turns out convenient to \emph{define}  
\begin{equation*}
	\Ihat_n(\omega) = I_n^Z(\omega) - \frac{8\tau_n^2}{\privpar^2},
\end{equation*}
which can be seen as an substitute of the quantity $I_n^{X}$.
Based on the definition of $\Ihat_n$ we can now proceed as in the non-private case.
For a fixed model $\mf$, we put
\begin{equation}\label{eq:def:fhat}
	\fhat_{\mf} = \sum_{i \in S_\mf} \widehat a_i \phi_i,
\end{equation}
where
\begin{equation}\label{eq:def:aihat}
\ahat_i = \int_{-\pi}^{\pi} \Ihat_n(\omega) \phi_i(\omega) \dd \omega.
\end{equation}

The following assumption is used in the proof of Theorem~\ref{thm:upper:minimax} to bound the probability of the event that $Z_i \neq Z_i'$ for at least one index $i$.
\begin{assumption}\label{ASS:SUBGAUSS}[Sub-Gaussianity, see Section~2.3 in \cite{boucheron2013concentration}]
	Let $\mu$ denote the (unknown) mean of the time series $X$.
	The marginals $X_t - \mu$ of the stationary time series $(X_t - \mu)_{t \in \Z}$ are sub-Gaussian with variance factor $\nu > 0$, that is, 
	\begin{equation*}
		\psi_{X_t - \mu}(\lambda) \leq \frac{\lambda^2 \nu}{2} \qquad \forall \lambda \in \R,
	\end{equation*}
	where $\psi_{X_t-\mu}(\lambda) = \log \Eb e^{\lambda (X_t- \mu)}$ denotes the logarithmic moment-generating function of the random variable $X_t-\mu$.
\end{assumption}

Note that we do not assume the mean $\mu$ to be known for our analysis.
A direct consequence of Assumption~\ref{ASS:SUBGAUSS} is the bound
\begin{equation}\label{EQ:CONC:INEQ:SUBGAUSS}
	\Pb( \lvert X_t -\mu \rvert > x) \leq 2 e^{-x^2/(2\nu)} \qquad \text{for all } x >0,
\end{equation}
see, for instance, \cite{boucheron2013concentration}, Theorem~2.1.
We will only need this bound for our further results.

\begin{theorem}[Upper bound]\label{thm:upper:minimax}
	Let Assumptions \ref{ASS:GAMMA_K} and \ref{ASS:SUBGAUSS} hold.
	Further assume that the model $\mf$ is given by a subspace $S_\mf$ of symmetric functions that satisfies $\lVert \phi_i \rVert_\infty \leq C \sqrt{n}$ for all $i \in \Ic_\mf$. 
	Let $Z_i$ be defined as in~\eqref{eq:def:Z_i} with $\tau_n^2 = 56\nu \log(n)$.
	Consider the estimator $\fhat_\mf$ defined through Equations \eqref{eq:def:fhat} and \eqref{eq:def:aihat}.
	Then,
	\begin{equation}\label{eq:upper}
		\Eb \lVert \fhat_\mf - f \rVert^2 \leq \lVert f - f_\mf \rVert^2 + C D_\mf (1 + \log(n)) \left[ \frac{1}{n} \vee \frac{\tau_n^4}{n\privpar^4} \right] 
	\end{equation}
	where $f_\mf$ denotes the projection of $f$ on the space $S_\mf$.
\end{theorem}

\begin{remark}
	Of course, if the time series $X$ is known to be bounded, say $\lvert X_t\rvert \leq K$ for all $t \in \Z$, the quantity $\tau_n$ in this section can be replaced with $K$ which removes at least some of the logarithmic factors (the ones arising via $\tau_n$) in the upper bound.
\end{remark}

\begin{remark}
	In the proof of Theorem~\ref{thm:upper:minimax}, Assumption~\ref{ASS:SUBGAUSS} is only needed to bound the probability of the event $\{ \exists i : X_i \neq \Xtilde_i \}$.
	For this purpose, the assumption of sub-Gaussianity might be replaced with assuming subexponential tails for the marginals.
	This would lead to a slightly different (but still logarithmic in terms of $n$) definition of the truncation threshold $\tau_n$.
	However, for Theorem~\ref{thm:adaptation} we will have to impose Gaussian marginals.
\end{remark}

\begin{remark}
	The quantity $\nu$ in Assumption~\ref{ASS:SUBGAUSS} is usually not given to the statistician but can be easily replaced by taking an estimator for this upper variance bound instead.
\end{remark}

\begin{example}[Sobolev ellipsoids and analytic functions]\label{ex:rates}
	In order to illustrate the upper bound~\eqref{eq:upper}, we consider the case where each model can be identified with a natural number: we have $\mf \in \N_0$, set $\Ic_\mf  = \llbracket -\mf,\mf \rrbracket$, and $S_\mf = \spn( \e_j )_{j \in \Ic_\mf}$ with $\e_j(\omega) = \exp(-\ii j \omega)$ denoting the (complex) Fourier basis functions.
	In terms of these basis functions, smoothness may be expressed by assuming membership of $f$ to an ellipsoid
	\begin{equation*}
		\Fc(\beta, L) = \left\{ f = \sum_{j \in \Z} f_j \e_j : f \geq 0 \text{ and } \sum_{j \in \Z} f_j^2 \beta_j^2 \leq L^2 \right\}
	\end{equation*}
	where $L > 0$ and $\beta = (\beta_j)_{j \in \Z}$ is a strictly positive symmetric sequence such that $\beta_0 = 1$ and $(\beta_n)_{n \in \N_0}$ is non-decreasing.
	Typical examples of sequences include the cases where $\beta_j \asymp \lvert j \rvert^s$ (Sobolev ellipsoids) and $\beta_j \asymp \exp(p \lvert j \rvert)$ for some $p \geq 0$ (class of analytic functions).
	Under our assumptions, the squared bias in the upper bound~\eqref{eq:upper} may be bounded as
	\begin{equation*}
		\lVert f_\mf - f \rVert^2 = \sum_{\lvert j \rvert > \mf} f_j^2 \leq \beta_\mf^{-2} \sum_{\lvert j \rvert > \mf} f_j^2 \beta_j^2 \leq L^2 \beta_\mf^{-2}.
	\end{equation*}
	Thus, the trade-off between squared bias and variance is equivalent to the best compromise between $\beta_\mf^{-2}$ and $(2\mf + 1) \cdot (1 + \log(n)) \left[ 1/n \vee \tau_n^4/(n\privpar^4) \right] $.
	In the polynomial case $\beta_j = \lvert j \rvert^s$, the best compromise is realized by choosing $\mfast \asymp \left[ (1 + \log(n)) \left( 1/n \vee \tau_n^4/(n\privpar^4) \right)  \right]^{-1/(2s+1)}$ leading to the rate $\left[ (1 + \log(n)) \left( 1/n \vee \tau_n^4/(n\privpar^4) \right)  \right]^{2s/(2s+1)}$.
	It is remarkable in the setup of spectral density estimation that also the part of the rate in terms of the privacy parameter $\privpar$ does not suffer from a loss in the exponent whereas in the setup of density estimation the optimal non-private rate $n^{-2s/(2s+1)}$ deteriorates to $n^{-2s/(2s+1)} \vee (n\privpar^2)^{-s/(s+1)}$ under differential privacy.
	In the case where $\beta_j = \exp(p\lvert j \rvert)$, we take $\mfast \asymp \log(n) + \log(\privpar)$ to obtain the rate $(\log(n) + \log(\privpar)) \cdot (1 + \log(n)) \left[ 1/n \vee \tau_n^4/(n\privpar^4) \right] $.
\end{example}

\subsection*{Lower bounds}

In this subsection, we derive minimax lower bounds for function classes that can be written as ellipsoids in terms of the Fourier coefficients of the function, that is, the classes $\Fc(\beta,L)$ introduced in Example~\ref{ex:rates}.
As discussed above, this general approach includes Sobolev ellipsoids and classes of analytic functions.
We determine both a non-private and and private lower bound, the former one valid already in the framework where a snippet from the original time series $X$ can be observed, the second one being special to the considered privacy scenario with observation $Z_{1:n}$.

\begin{theorem}[Lower bound]\label{thm:lower}
	Assume that the time series $X$ is Gaussian, and consider the class $\Fc(\beta, L)$ of potential spectral densities introduced in Example~\ref{ex:rates}.
	Further assume that anonymized data $Z_{1:n}$ are generated via a (potentially interactive) channel $Q$ ensuring local differential privacy.
	\begin{enumerate}[a)]
		\item\label{it:nonprivate:lower}
		Assume that $B \defeq \sum_{j \in \Z} \beta_j^{-2} < \infty$.
		Define $\knast$ and $\Psi_n$ via
		\begin{align*}
			 \knast &= \argmin_{k \in \N} \left[ \max \left(  \beta_k^{-2}, \frac{2k+1}{n} \right)  \right],\\
			 \Psi_n &= \max \left(  \beta_{\knast}^{-2}, \frac{2\knast+1}{n} \right),
		\end{align*}
		and assume that there is a positive constant $\eta$ such that
		$$0 < \eta^{-1} \leq \Psi_n^{-1} \min \left\{
		\beta_{\knast}^{-2} , \frac{2\knast + 1}{n} \right\}.$$
		Then,
		\begin{equation*}
			\inf_{\ftilde} \sup_{f \in \Fc(\beta, L)} \Eb \lVert \ftilde - f \rVert^2 \gtrsim \frac{2\knast + 1}{n}
		\end{equation*}
		holds where the infimum is taken over all estimators $\ftilde$ of $f$ based on the privatized sample $Z_{1:n}$.
		\item\label{it:private:lower}
		It holds
		\begin{equation*}
		\inf_{\ftilde} \sup_{f \in \Fc(\beta, L)} \Eb \lVert \ftilde - f \rVert^2 \gtrsim  \min \left\{  \frac{\pi}{n(e^\privpar - 1)^2} , \frac{L^2}{4}  \right\},
		\end{equation*}
		where the infimum is taken over all estimator $\ftilde$ of $f$ based on the privatized sample $Z_{1:n}$.
	\end{enumerate}
\end{theorem}

\begin{remark}
	The proof of statement~\ref{it:nonprivate:lower} of Theorem~\ref{thm:lower} is based on a reduction to estimators in terms of the original sample $X_{1:n}$.
	Indeed, any lower bound valid for estimators in terms of the original sample stays valid in the privacy case since working with differentially private data can equivalently be interpreted as restricting the set of potentially available estimators from the set of all measurable functions in terms of $X_{1:n}$ to the set of functions of the form $\ftilde \circ Q$ where $Q$ is a channel that yields differential privacy and $\ftilde$ any measurable function in terms of $Z_{1:n}$.
	In the appendix, we give the complete proof since we were not able to find a good reference in the existing literature (the articles \cite{bentkus1985rate} and \cite{efromovich1998data} consider different function classes).
\end{remark}

By combining the non-private and the private lower bound we directly obtain the following corollary.
\begin{corollary} Under the Assumptions of Theorem~\ref{thm:lower} we have
	\begin{equation*}
		\inf_{\ftilde} \sup_{f \in \Fc(\beta, L)} \Eb \lVert \ftilde - f \rVert^2 \gtrsim \max \left\{ \Psi_n^2 , \min \left\{ 1, \frac{1}{n(e^\alpha - 1)^2}  \right\} \right\}.
	\end{equation*}
\end{corollary}

\begin{remark}
	Up to logarithmic factors the lower bounds determined coincide with the given upper bounds.
	However, our results here do not give a complete answer concerning the exact dependence of the optimal convergence rate in terms of the privacy parameter $\privpar$.
	Intuitively, part \ref{it:private:lower} states only the loss that can be explained from the constant basis function when the spectral density is written in terms of the trigonometric basis.
	It can already be seen here that a deterioration of the usual rate (given by part \ref{it:nonprivate:lower}) is unavoidable if $\privpar$ is too small.
	In this case, one can obtain a lower bound by comparing distributions characterized by two different but constant spectral densities (see the proof of part \ref{it:private:lower} in the appendix).
	Then, there is no dependence between the $X_t$, that is, we have access to an i.i.d. sample and the well-known information theoretic inequalities for differential privacy from the paper \cite{duchi2018minimax} are available.
	These data processing inequalities do not longer hold for dependent $X_t$.
	Developing tools in this direction that help to understand the exact scaling behaviour represent an interesting point of departure for further investigations.
	Note also that even for the Fourier coefficient associated with the constant basis function we do not have coincidence for the scaling in terms of $\privpar$: we have a term $1/(n(e^\privpar - 1)^2)$ in the lower bound (which behaves as $1/(n\privpar^2)$ for small $\privpar$) but a term of order $1/(n\privpar^4)$ in the upper bound (plus extra logarithmic factors).
	This last issue might be tackled by publishing an anonymized version of $X_t^2$ in addition to the privatized version of $X_t$ since for computation of the empirical correlation coefficient associated with the constant basis function no interaction between the data holders is necessary.
\end{remark} 
\section{Risk bound for the adaptive estimator}\label{s:adaptation}

In Section~\ref{s:minimax} we have derived the upper risk bound \eqref{eq:upper} for fixed models $\mf$.
The near optimality for the class of Sobolev ellipsoids was equally illustrated in Example~\ref{ex:rates} and the accompanying lower bounds established in Theorem~\ref{thm:lower}.
The performance of the rate optimal estimators hinges on the choice of a suitable approximating model $\mf$ the choice of which depends on both the sample size $n$ and the regularity of the functions in the considered function class.
Since such regularity assumptions are usually not realistic to be fulfilled, there is need to obtain a suitable model in completely data-driven way.

In order to define the adaptive estimator, first introduce the contrast
\begin{equation*}
	\Upsilon_n(t) = \int_{-\pi}^{\pi} t^2(\omega) \dd \omega - 2 \int_{-\pi}^{\pi} t(\omega) \Ihat_n(\omega) \dd \omega.
\end{equation*}
Note that, the estimator $\fhat_\mf$ associated with the fixed model $\mf$ in Section~\ref{s:minimax} satisfies
\begin{equation*}
	\Upsilon_n(\fhat_\mf) = \min_{t \in \Sc_\mf} \Upsilon_n(t). 
\end{equation*}
The model selection step is performed by putting
\begin{equation}\label{eq:def:mhat}
	\mfhat = \argmin_{\mf \in \Mc_n} \, \{ \Upsilon_n(\fhat_\mf) + \pen(\mf) \}
\end{equation}
where $\Mc_n$ is some set of potential models, and $\pen \colon \Mc_n \to [0,\infty)$ a penalty function is given by
\begin{equation}\label{eq:def:pen}
	\pen(\mf) = C  D_\mf \max \left\lbrace \frac{1}{n}, \frac{\tau_n^4}{n\privpar^4} \right\rbrace \cdot (L_{\mf}^4 + L_\mf + \log(n) ) \cdot (1 + \lVert f \rVert_\infty)^2
\end{equation}
for some constant $C > 0$ that has to be chosen large enough.
Finally, the adaptive estimator of the spectral density $f$ is defined as
\begin{equation*}
	\ftilde = \fhat_{\mfhat}.
\end{equation*}

Before we can state our main result, we have to impose the following assumptions on the collection $\Mc_n$ of models.
These assumptions are already present in the work of \cite{comte2001adaptive}, and no extra assumptions on the models are needed in the privacy framework.

\begin{assumption}\label{ass:model:dim}
	Each $\Sc_\mf$ is a linear finite-dimensional subspace of $L^2([-\pi,\pi])$ containing symmetric functions with dimension $\dim(S_\mf) = D_\mf \geq 1$.
	Moreover, $D_n \defeq \max_{\mf \in \Mc_n} D_\mf \leq n$.
\end{assumption} 

\begin{assumption}\label{ass:rbar}
	Let $(\phi_i)_{i \in \Ic_{\mf}}$ be an orthonormal basis of $S_\mf$, and $\beta = (\beta_i)_{i \in \Ic_{\mf}} \in \R^{D_\mf}$. Set $\lvert \beta \rvert_\infty = \sup_{i \in \Ic_{\mf}} \lvert \beta_i \rvert$.
	Then, for all $\mf \in \Mc_n$,
	\begin{align*}
	\rbar_{\mf} \defeq \frac{1}{\sqrt{D_\mf}} \sup_{\beta \neq 0} \frac{\lVert \sum_{i \in \Ic_{\mf}} \beta_i \phi_i \rVert_\infty}{\lvert \beta \rvert_\infty} \leq C_{\rbar} \sqrt{\frac{n}{D_\mf}}.
	\end{align*}
\end{assumption}

\begin{assumption}\label{ass:weights}
	$\sum_{\mf \in \Mc_n} e^{-L_\mf D_\mf} \leq C_L < \infty$ for some positive weights $L_\mf$.
\end{assumption}

Remarks 2.3--2.6 from \cite{comte2001adaptive} show that Assumptions~\ref{ass:rbar} and \ref{ass:weights} are satisfied for the models mentioned in the introduction for suitable values $C_{\rbar}$ and $L_{\mf}$. 

\begin{theorem}\label{thm:adaptation}
	Let Assumption \ref{ASS:GAMMA_K}  hold.
	Let $Z_i$ be defined as in~\eqref{eq:def:Z_i} with $\tau_n^2 = 56\nu \log(n)$.
	Consider the estimator $\ftilde = \fhat_{\mfhat}$ defined through Equations \eqref{eq:def:fhat}, \eqref{eq:def:aihat}, and \eqref{eq:def:mhat} where the penalty function is defined in \eqref{eq:def:pen}.
	Let Assumptions~\ref{ass:model:dim}--\ref{ass:weights} hold.
	Then, 
	\begin{equation*}
		\Eb \lVert \ftilde - f \rVert^2 \lesssim \inf_{\mf \in \Mc_n} [ \lVert f - f_\mf \rVert^2 + \pen(\mf) ]\\
		+ C(C_{\rbar},\lVert f \rVert_\infty) \max \left\lbrace \frac{1}{n}, \frac{\tau_n^4}{n^3 \privpar^4} \right\rbrace.
	\end{equation*}
\end{theorem}

\begin{remark}
	Unfortunately, the definition of the penalty function introduced above depends on the unknown value $\lVert f \rVert_\infty$.
	In practise, one can replace this quantity by an appropriate estimator. Theoretical results can be proved for this more realistic estimator as well.
	We do not realize this here, and refer the interested reader to the papers \cite{comte2001adaptive} and \cite{kroll2019nonparametric} where this idea has been put into practise.
	The resulting fully-adaptive estimator can be shown to satisfy an oracle inequality as in the case of known $\lVert f \rVert_\infty$ under mild assumptions.
\end{remark} 
\section{Numerical study}\label{s:sim}

In this section, we illustrate our findings by a small simulation study.
The code that can be used to (re)produce the results is available under
\begin{center}\scriptsize
	\url{https://gitlab.com/kroll.martin/adaptive-private-spectral-density-estimation}.
\end{center}
We consider the same time series model as \cite{neumann1996spectral} and \cite{comte2001adaptive}, that is,  we consider the time series $(X_t)_{t \in \Z}$ defined as
\begin{equation*}
	X_t = X_t^{\text{ARMA}} + \sigma X_t^\text{WN}
\end{equation*}
where $X_t^{\text{ARMA}}$ is an ARMA(2,2)-process,
\begin{equation*}
	X_t^{\text{ARMA}} + a_1 X_{t-1}^{\text{ARMA}} + a_2 X_{t-2}^{\text{ARMA}} = b_0 \epsilon_t + b_1 \epsilon_{t-1} + b_2 \epsilon_{t-2},
\end{equation*}
and $(\epsilon_t)_{t \in \Z}$ and $(X^\text{WN}_t)_{t \in \Z}$ are independent Gaussian white noise processes with unit variance.
From the cited papers we also adopt the choices of the parameters ($a_1=0.2$, $a_2=0.9$, $b_0=1$, $b_1=0$, $b_2 = 1$, and $\sigma = 0.5$).
We consider time series snippets of length $n \in \{ 10000, 20000\}$ and simulate $T=100$ replications of each setup.
In contrast to the mentioned papers, which consider a non-private framework, our principal aim is to illustrate the effect of the privacy level $\privpar$.
For this purpose, we consider $\alpha \in \{ +\infty, 5, 2.5 \}$ where formally putting $\alpha = +\infty$ corresponds to the case without any privacy constraint.
Note that these choices of the privacy parameter are very conservative and provide only a moderate anonymization of the data (see, for instance, Figure~3 in \cite{duchi2018minimax} where the link between the privacy parameter and a hypothesis testing problem is illustrated).

For each parameter setup, we computed the mean $L^2$-risk over the $T=100$ replications, its standard deviation $\vhat$, and the $\pm 95\%$ confidence intervals defined as $1.96\vhat/\sqrt{T}$ (see \cite{comte2001adaptive,neumann1996spectral}).
We slightly modified the method considered in Sections~\ref{s:minimax}
and \ref{s:adaptation} for the theoretical analysis in order to perform our simulation experiments.
Instead of a logarithmically increasing sequence $\tau_n$ (which was principally introduced to control the probability of the event $\Acomplement$ introduced in the analysis in the appendix), we took $\tau_n = 4$ after some calibrations.
As \cite{comte2001adaptive}, we restricted ourselves to histogram estimators of the spectral density.
For a given dimension $D_{\mf} = d$, the orthonormal basis functions are defined as
\begin{equation*}
	\phi_j^{(d)} = \sqrt{\frac{d}{\pi}} \1_{[\pi j /d, \pi(j+1/d)}, \qquad j \in \llbracket 0,d-1\rrbracket
\end{equation*}
(we define the basis functions only on $[0,\pi)$ and extend the final estimator on the interval $[-\pi,\pi]$ by exploiting the symmetry of the target spectral density).
For the model $\mf$, the estimator $\fhat_\mf$ is then
\begin{equation*}
	\fhat_\mf = \sum_{j=0}^{d-1} \ahat_j^{(d)} \phi_j^{(d)}
\end{equation*}
where the estimated coefficients are calculated via the formula
\begin{equation*}
	\ahat_j^{(d)} = \sqrt{\frac{d}{\pi}} \left[ \frac{c_0}{2d} + \frac{1}{\pi} \sum_{r=1}^{n-1} \frac{c_r}{r} \left( \sin \left( \frac{\pi(j+1)r}{d} \right)  - \sin \left( \frac{\pi j r}{d} \right)  \right) \right] 
\end{equation*}
for $j \in \llbracket 0,d-1\rrbracket$ where $c_r = c_{r,n}$, $r \in \llbracket 0,n-1 \rrbracket$ are the empirical covariances of the masked data $Z_{1:n}$, that is,
\begin{equation*}
	c_{r,n} = \frac{1}{n} \sum_{k=1}^{n-r} (Z_k - \Zbar_n) (Z_{k+r} - \Zbar_n)
\end{equation*}
(the value $c_{0,n}$ has to be modified by subtracting $8\tau_n^2/\privpar^2$ afterwards).
Ignoring logarithmic factors and constants in the theoretical penalty in Section~\ref{s:adaptation}, this leads to the following form of the penalized contrast criterion:
\begin{equation*}
	- \sum_{j=0}^{d-1} (\ahat_j^{(d)})^2 + \frac{\kappa d}{n} \max \left\{ 1, \frac{\tau_n^4}{\alpha^4} \right\}
\end{equation*}
(we took $\kappa = 1$).
We minimized this criterion over potential dimensions $d\in \llbracket 1,50 \rrbracket$.
The results of our simulation study are summarized in Table~\ref{table:sim} and illustrated in Figures \ref{fig:alpha:inf}, \ref{fig:alpha:5_0}, and \ref{fig:alpha:2_5} (note the different scaling of the $y$-axes in the plots).
A profound loss of performance is encountered for decreasing values of $\privpar$ which can be compensated with taking considerably longer snippets from the time series only.
This might make inference from privatized data difficult in applications where only samples of moderate size can be collected. 

	\begin{table}
		\renewcommand{\arraystretch}{1.5}
		\caption{Results of the simulation study. The table contains the mean of the $L^2$-risk over $T=100$ replications of the experiment, and the $\pm 95\%$ intervals computed as in \cite{neumann1996spectral} as $1.96\widehat v/\sqrt{T}$ where $\widehat v$ is the standard deviation.}\label{table:sim}
		\footnotesize
		\begin{center}
			\begin{tabular}[t]{lccc!{\color{gray}\vrule}ccc}
						& \multicolumn{3}{c!{\color{gray}\vrule}}{n = 10000} & \multicolumn{3}{c}{n = 20000} \\
						$\privpar$ & $+\infty$ & $5.0$ & $2.5$ & $+\infty$ & $5.0$ & $2.5$\\
						\hline
						$L^2$-risk & 0.00216 & 0.01316 & 0.13629 & 0.00159 & 0.00734 & 0.07126\\
						$\pm $ 95\% CI & 0.00012 & 0.00048 & 0.00464 & 0.00007 & 0.00022 & 0.00243 
			\end{tabular}
		\end{center}
		\normalsize
	\end{table}

\begin{center}
	\begin{figure}
		\centering\textbf{\boldmath$\privpar = + \infty$}\par \begin{subfigure}[c]{0.45\textwidth}
			\includegraphics[width=0.99\textwidth]{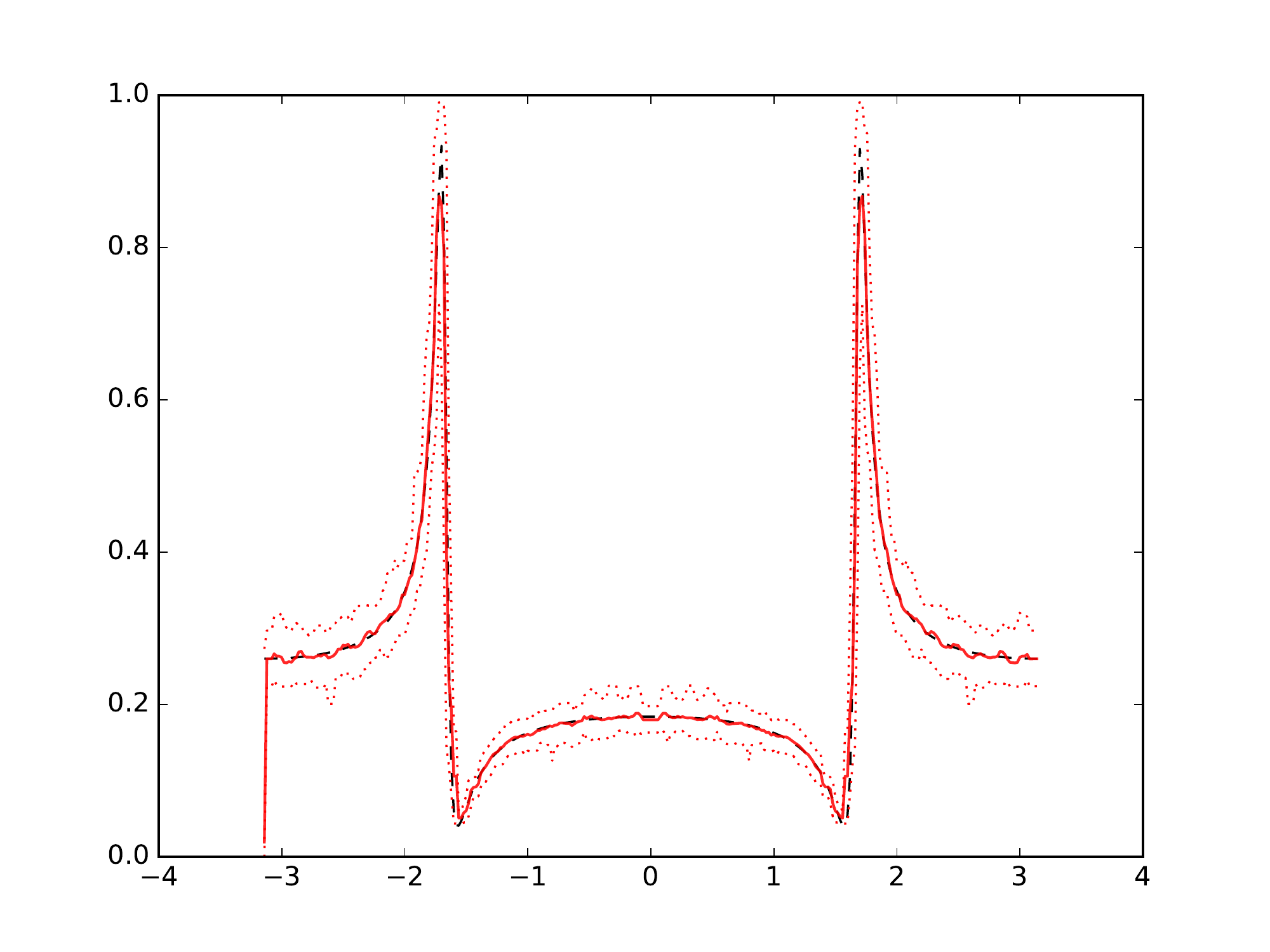}
			\subcaption{Snippet length $n = 10000$}
		\end{subfigure}
		\begin{subfigure}[c]{0.45\textwidth}
			\includegraphics[width=0.99\textwidth]{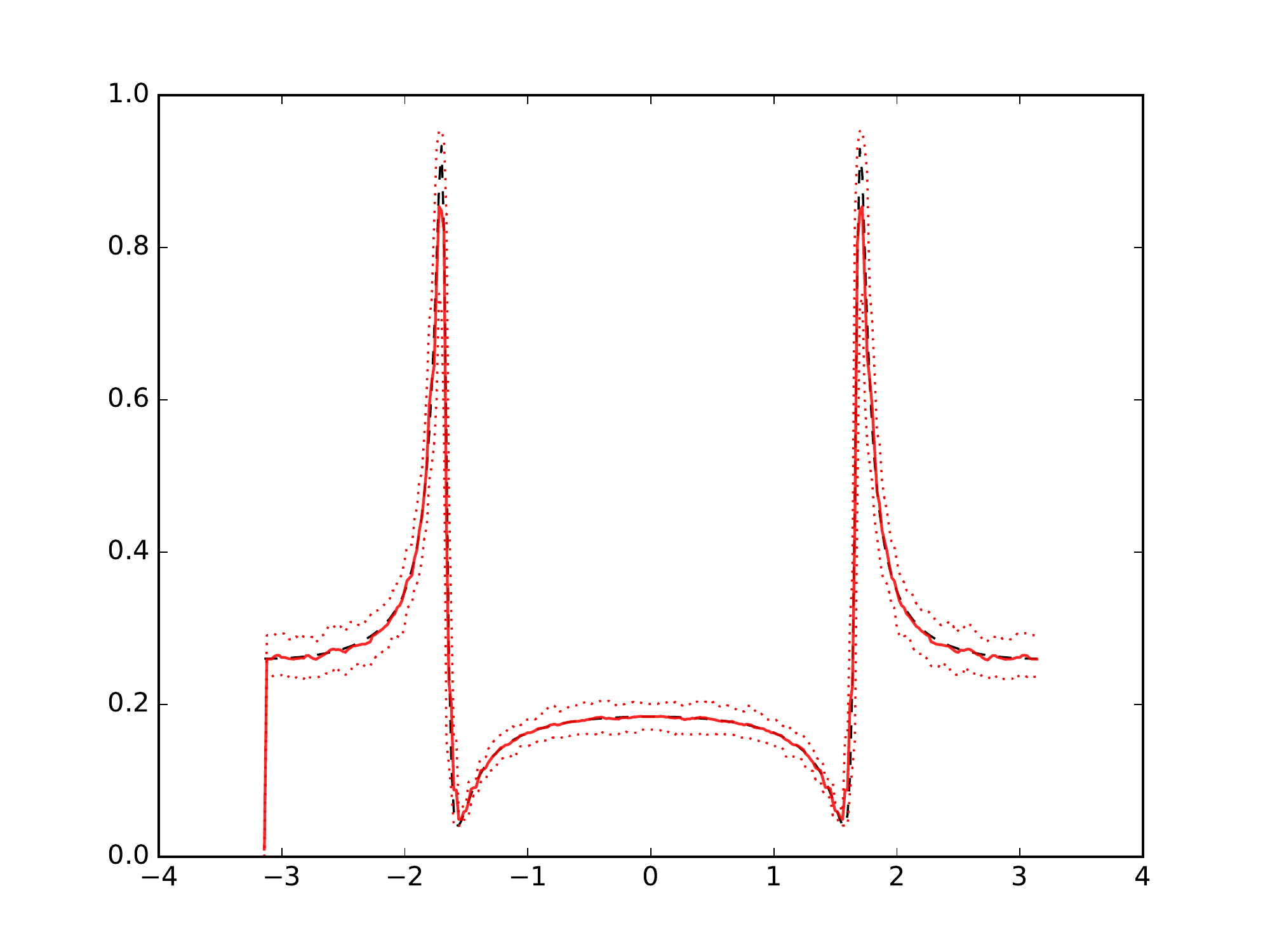}
			\subcaption{Snippet length $n = 20000$}
		\end{subfigure}
		\caption{The figures show for the two considered snippet sizes $n \in \{ 10000, 20000\}$ the mean of the estimator (red solid line) and both the 0.95 and 0.05 pointwise quantile (red dotted lines) over $T=100$ replications for the case $\privpar = + \infty$ (this corresponds to the case without privacy constraints). The true spectral density is represented as a black dashed line.}\label{fig:alpha:inf}
	\end{figure}
\end{center}

\begin{figure}
		\centering\textbf{\boldmath$\privpar = 5.0$}\par \begin{subfigure}[c]{0.45\textwidth}
			\includegraphics[width=0.99\textwidth]{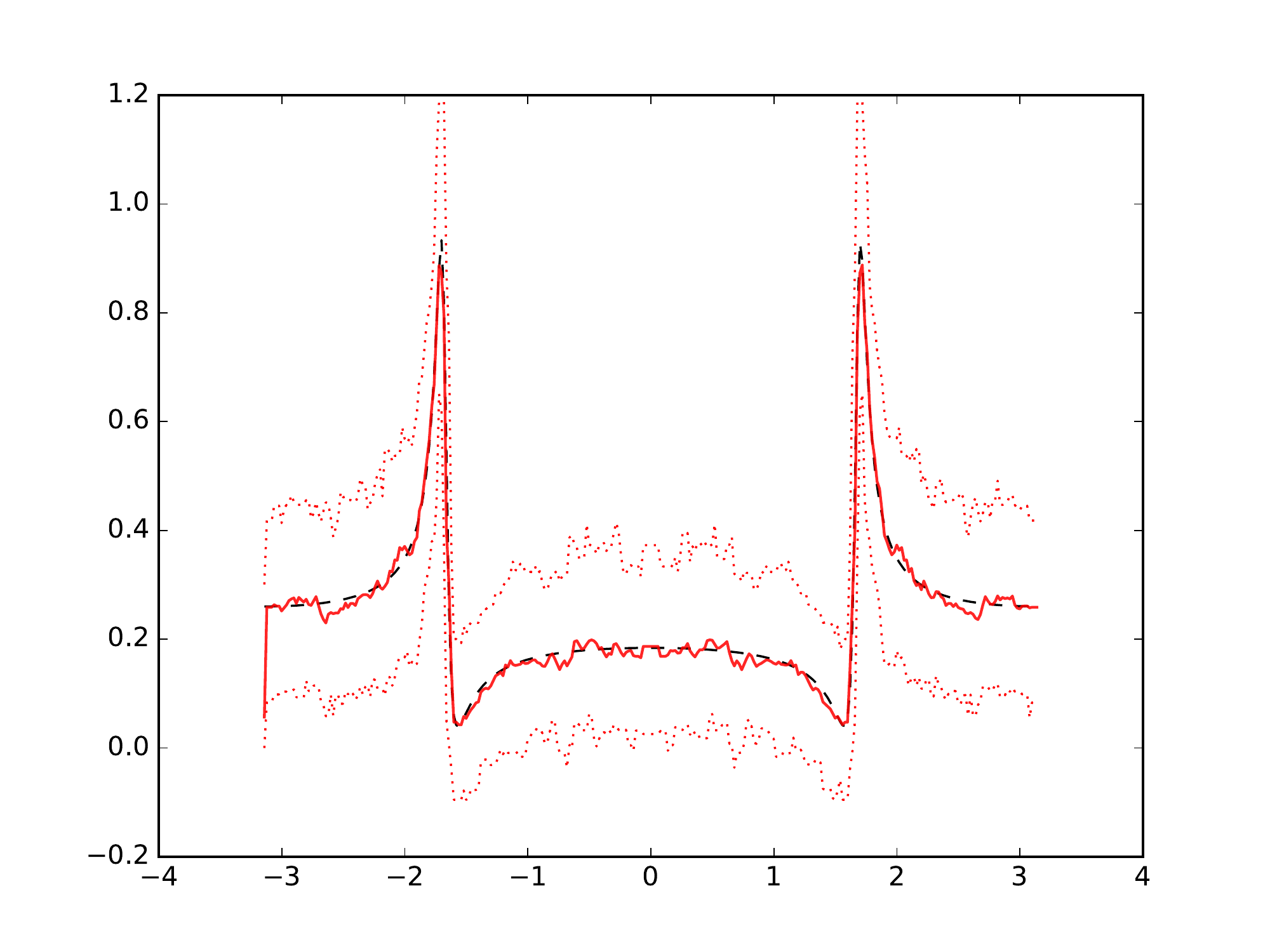}
			\subcaption{Snippet length $n = 10000$}
		\end{subfigure}
		\begin{subfigure}[c]{0.45\textwidth}
			\includegraphics[width=0.99\textwidth]{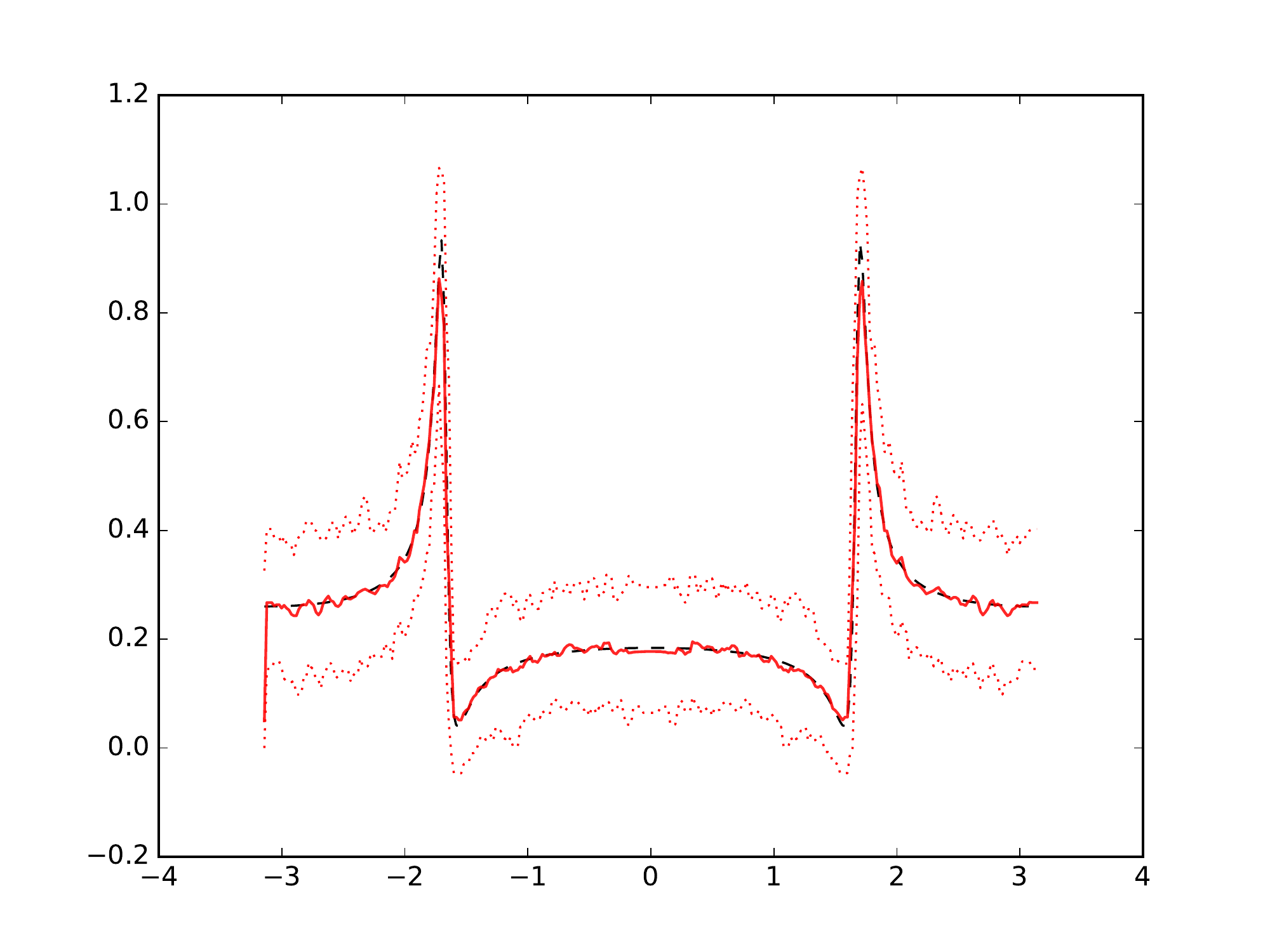}
			\subcaption{Snippet length $n = 20000$}
		\end{subfigure}
		\caption{The figures show for the two considered snippet sizes $n \in \{ 10000, 20000\}$ the mean of the estimator (red solid line) and both the 0.95 and 0.05 pointwise quantile (red dotted lines) over $T=100$ replications for the case $\privpar = 5.0$. The true spectral density is represented as a black dashed line.}\label{fig:alpha:5_0}
	\end{figure}

	\begin{figure}
		\centering\textbf{\boldmath$\privpar = 2.5$}\par \begin{subfigure}[c]{0.45\textwidth}
			\includegraphics[width=0.99\textwidth]{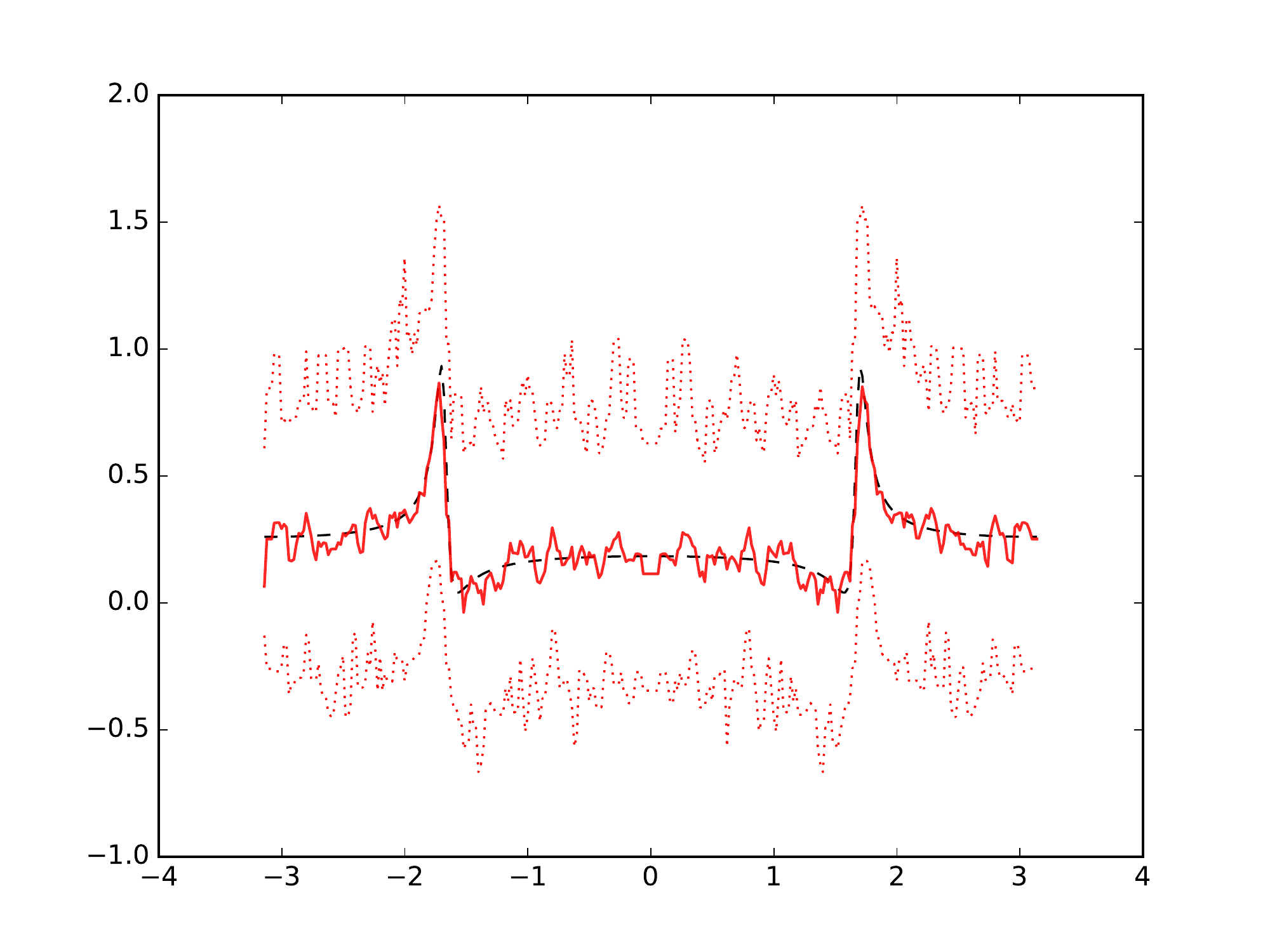}
			\subcaption{Snippet length $n = 10000$}
		\end{subfigure}
		\begin{subfigure}[c]{0.45\textwidth}
			\includegraphics[width=0.99\textwidth]{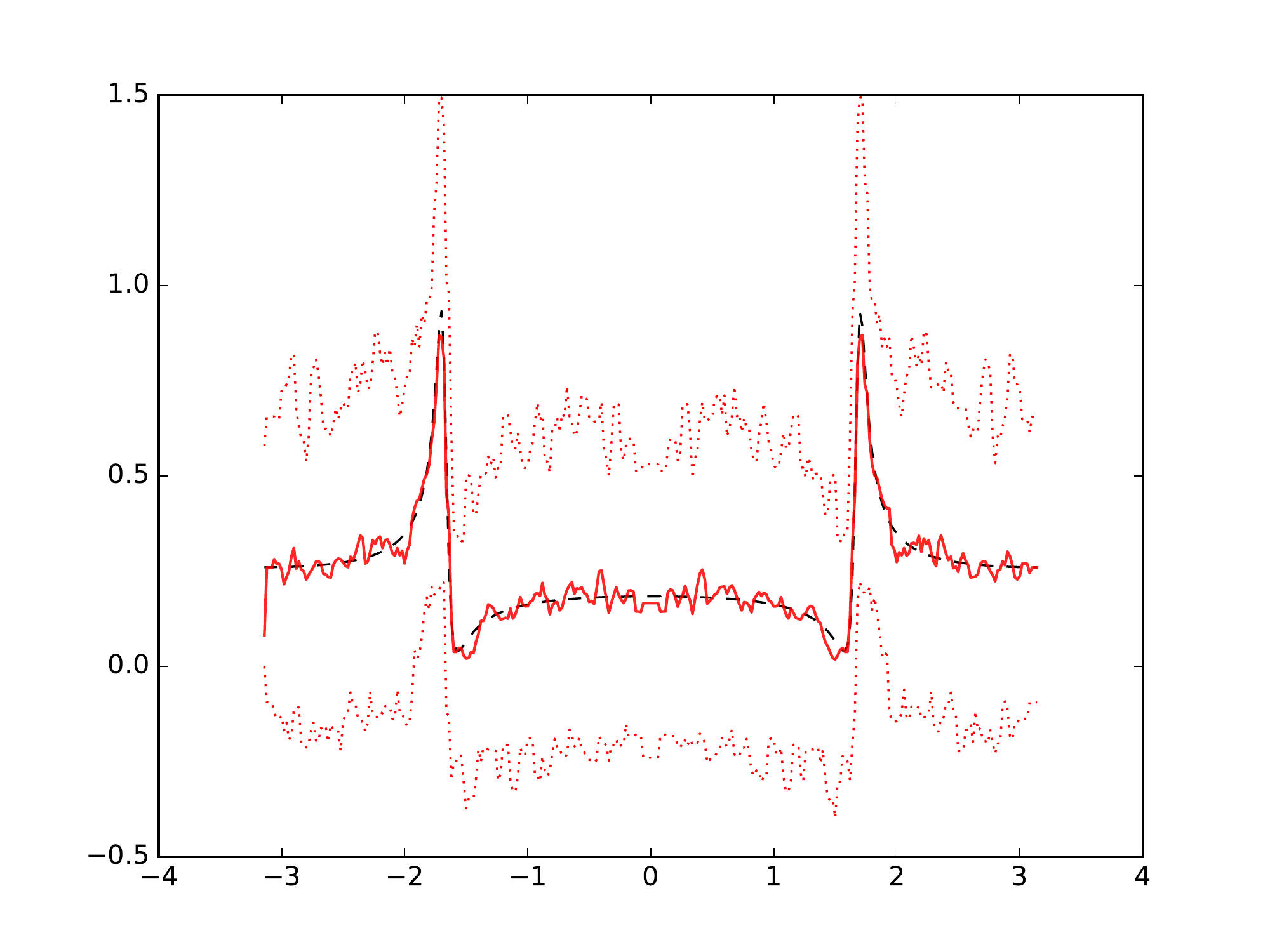}
			\subcaption{Snippet length $n = 20000$}
		\end{subfigure}
		\caption{The figures show for the two considered snippet sizes $n \in \{ 10000, 20000\}$ the mean of the estimator (red solid line) and both the 0.95 and 0.05 pointwise quantile (red dotted lines) over $T=100$ replications for the case $\privpar = 2.5$. The true spectral density is represented as a black dashed line.}\label{fig:alpha:2_5}
	\end{figure} 
\section{Summary and outlook}\label{s:discussion}

In this paper, we have extended the model selection approach for adaptive nonparametric spectral density estimation to the framework of local $\privpar$-differential privacy.
We were able to derive an oracle inequality similar to the one in the non-private setup.
Since the proposed adaptive procedure is limited to Gaussian time series it might also be of interest to study whether known adaptive estimators that work in non-Gaussian frameworks (for instance, the wavelet estimator considered in \cite{neumann1996spectral}) can also be transferred to the framework of the present paper.
The exact dependence of minimax rates of convergence on the privacy parameter as well as the unclear necessity of logarithmic factors in these rates is a remaining open problem that hopefully stimulates the development of further theoretical results.
In addition, a more detailed series of simulation experiments seems to be necessary in order to calibrate an estimator that produces reliable results in practise.
 
\appendix

\section{Proofs of Section~\ref{s:minimax}}

The following result (which is valid without any distributional assumptions on the stationary time series) has been proven in \cite{comte2001adaptive}.
\begin{proposition}[\cite{comte2001adaptive}, Proposition~1]\label{prop:ex:comte}
	Let $X$ be a stationary sequence with autocovariance function satisfying Assumption~\ref{ASS:GAMMA_K}.
	Then
	\begin{equation*}
		\int_{-\pi}^{\pi} (f(\omega) - \Eb (I_n(\omega)))^2 \dd \omega \leq \frac{M_1 + 39M^2}{2\pi n} \eqdef \frac{M_2}{n}.
	\end{equation*}
\end{proposition}
This result can also be applied to the time series $Z'$. Then the constant $M_1$ does not change but for the constant $M$ we have $M^{Z'} = M^X + 8\tau_n^2/\privpar^2$.

\subsection{Proof of Theorem~\ref{thm:upper:minimax} (Upper bound for fixed model $\mf$)}
	
Let us introduce the event $A$ and its complement defined as follows:
\begin{equation*}
	A = \bigcap_{i=1}^n \{ \Xtilde_i = X_i \}, \qquad \Acomplement = \bigcup_{i=1}^n \{\Xtilde_i \neq X_i \}.
\end{equation*}
As above, let us denote with $f_\mf$ the projection of $f$ on the space $S_\mf$.
We have the decomposition
\begin{align}
	\Eb \lVert \fhat_\mf - f \rVert^2 &= \lVert f_\mf - f \rVert^2 + \Eb \lVert \fhat_\mf - f_\mf \rVert^2\notag\\
	&= \lVert f_\mf - f \rVert^2 + \Eb \lVert \fhat_\mf - f_\mf \rVert^2 \1_A + \Eb \lVert \fhat_\mf - f_\mf \rVert^2 \1_{A^\complement}.\label{eq:decomp:upper}
\end{align}

The first (pure bias) term on the right-hand side is already in the form of the statement of the theorem, and we have to study the terms including $\1_A$ and $\1_{A^\complement}$ only.

\smallskip

\underline{Bound for $\Eb \lVert \fhat_\mf - f_\mf \rVert^2 \1_A$}:
By the very definition of $A$ we have $X_i=\Xtilde_i$ on $A$, and hence $Z_i = Z'_i = X_i + \xi_i$ for $\xi_i \sim \Laplace(2\tau_n/\privpar)$.
Hence, on the event $A$ the identity
\begin{equation*}
	I_n^Z(\omega) = I_n^{Z'}(\omega)
\end{equation*}
holds (with $I_n^{Z'}$ defined exactly as $I_n^Z$ with $Z$ replaced with $Z'$), and we have
\begin{align*}
	\lVert \fhat_\mf - f_\mf \rVert^2 \1_A &= \sum_{i \in \Ic_\mf} \lvert \langle f - \Ihat_n, \phi_i \rangle \rvert^2 \1_A\\
	&= \sum_{i \in \Ic_\mf} \lvert \langle f - (I_n^Z(\omega) - \frac{8\tau_n^2}{\privpar^2}), \phi_i \rangle \rvert^2 \1_A\\
	&= \sum_{i \in \Ic_\mf} \lvert \langle f - (I_n^{Z'}(\omega) - \frac{8\tau_n^2}{\privpar^2}), \phi_i \rangle \rvert^2 \1_A\\
	&\leq \sum_{i \in \Ic_\mf} \lvert \langle f - ( I_n^{Z'}(\omega) - \frac{8\tau_n^2}{\privpar^2} ) ,\phi_i \rangle \rvert^2\\
	&= \sum_{i \in \Ic_\mf} \lvert \langle f^{Z'} - I_n^{Z'}, \phi_i \rangle \rvert^2.
\end{align*}
where the last identity is established in \eqref{eq:fZprime}.
From this we get
\begin{align*}
	\lVert \fhat_\mf - f_\mf \rVert^2 \1_A &\leq 2 \sum_{i \in \Ic_\mf} (\lvert \langle f^{Z'} - \Eb I_n^{Z'},\phi_i \rangle \rvert^2 + \lvert \langle \Eb I_n^{Z'} - I_n^{Z'}, \phi_i \rangle \rvert^2 )\\
	&= 2 \lVert (f^{Z'} - \Eb I_n^{Z'})_\mf \rVert^2 + 2 \sum_{i \in \Ic_\mf} \lvert \langle \Eb I_n^{Z'} - I_n^{Z'}, \phi_i \rangle \rvert^2\\
	&\leq 2 \lVert f^{Z'} - \Eb I_n^{Z'} \rVert^2 + 2 \sum_{i \in \Ic_\mf}  \lvert \langle \Eb I_n^{Z'} - I_n^{Z'}, \phi_i  \rangle \rvert^2
\end{align*}
in order to bound the term $\lVert f^{Z'} - \Eb I_n^{Z'} \rVert^2$, we use Proposition~\ref{prop:ex:comte} in order to obtain:
\begin{align*}
	\lVert f^{Z'} - \Eb I_n^{Z'} \rVert^2 &= \int_{-\pi}^{\pi} (f^{Z'} - \Eb I_n(\omega))^2 \dd \omega \leq \max \left( \frac{M_1}{\pi n}, \frac{39 (M^{Z'})^2}{\pi n} \right).
\end{align*}
Note that Assumption~\ref{ASS:GAMMA_K} can also be applied to the time series $Z'$ instead of $X$ with $M_1= M_1^{Z'} = M_1^X$ and with $M=M^X$ replaced with $M^{Z'} = M^X + \frac{8\tau_n^2}{\privpar^2}$.
Hence,
\begin{equation}\label{eq:upper:A:1}
	\lVert f^{Z'} - \Eb I_n^{Z'} \rVert^2 \lesssim \max \left( \frac{\tau_n^4}{n \privpar^4} , \frac{1}{n} \right).
\end{equation}
Let us now consider the expression $\Eb \sum_{i \in \Ic_\mf} \lvert \langle \Eb I_n^{Z'} - I_n^{Z'}, \phi_i \rangle \rvert^2$.
We write
\begin{equation*}
	I_n^{Z'}(\omega)  = I_n^X + I_n^\xi  + \Itilde_n,
\end{equation*}
where
\begin{align*}
	I_n^X &= \frac{1}{2\pi n} \left\lvert \sum_{t=1}^n (X_t - \Xbar_n) e^{-\ii t \omega} \right\rvert^2,\\
	I_n^\xi &= \frac{1}{2\pi n} \left\lvert \sum_{t=1}^n (\xi_t - \xibar_n) e^{-\ii t \omega} \right\rvert^2, \qquad \text{and}\\
	\Itilde_n &= \frac{1}{2\pi n} \left( \sum_{t=1}^n (X_t - \Xbarn) e^{-\ii t \omega} \right) \left( \sum_{t=1}^n (\xi_t - \xibarn) e^{\ii t \omega} \right) +\\
	\hspace{1em}&+ \frac{1}{2\pi n} \left( \sum_{t=1}^n (X_t - \Xbarn) e^{\ii t \omega} \right) \left( \sum_{t=1}^n (\xi_t - \xibarn) e^{-\ii t \omega} \right).
\end{align*}
Hence, by exploiting that $\Eb \Itilde_n = 0$, we obtain
\begin{align*}
	\sum_{i \in \Ic_\mf} \lvert \langle \Eb I_n^{Z'} - I_n^{Z'}, \phi_i \rangle \rvert^2 &\leq  \sum_{i \in \Ic_\mf} \lvert \langle \Eb I_n^{X} - I_n^{X}, \phi_i \rangle \rvert^2\\
	&\hspace{1em}+ \sum_{i \in \Ic_\mf} \lvert \langle \Eb I_n^{\xi} - I_n^{\xi}, \phi_i \rangle \rvert^2 + \sum_{i \in \Ic_\mf} \lvert \langle \Eb \Itilde_n - \Itilde_n,  \phi_i \rangle \rvert^2.
\end{align*}
Put
\begin{align*}
G_{i,X}(\mf) &= \sup_{\zeta \in \{ \pm 1\} } \langle I_n^X - \Eb I_n^X,\zeta \phi_i  \rangle,\\
G_{i,\xi}(\mf) &= \sup_{\zeta \in \{ \pm 1\}} \langle I_n^\xi - \Eb I_n^\xi, \zeta \phi_i \rangle,\\
\Gtilde_i(\mf) &= \sup_{\zeta \in \{ \pm 1\}} \langle \Itilde_n, \zeta \phi_i \rangle.
\end{align*}
Then, for any constant $\kappa_X > 0$, we have
\begin{align*}
	\Eb \langle \Eb I_n^{X} - I_n^{X}, \phi_i \rangle ^2 &\leq \Eb \left[ \left( (G_{i,X}(\mf))^2 - \frac{4\kappa_X \lVert f \rVert_\infty^2  (1 + C_{\rbar}^2)}{n} \right)_+ \right]\\
	&\hspace{1em} + \frac{4\kappa_X \lVert f \rVert_\infty^2  (1 + C_{\rbar}^2)}{n}.
\end{align*}
Hence, by Lemma~\ref{l:E:GX}\footnote{Admittedly, using Lemmata~\ref{l:E:GX}, \ref{l:E:Gxi}, and \ref{l:E:Gtilde} here is like using a sledgehammer to crack a nut. At least for the term containing $I_n^X$ we can directly refer to p.~294 in \cite{comte2001adaptive} for an alternative reasoning. For the other terms, one could perform in the same manner with some tedious calculations but we do currently not see how one could establish an upper bound without any logarithmic terms and a better dependence on $\privpar$ than in our current estimate. Note that instead of Assumption~\ref{ass:rbar} we only need to assume that $\lVert \phi_i \rVert_\infty \leq C\sqrt{n}$ for the an orthonormal basis $(\phi_i)_{i \in \Ic_\mf}$ of the considered model. In addition, we can also put $L_\mf = 1$ here since in contrast to the proof of Theorem~\ref{thm:adaptation} no summation over all potenial models is performed.} we get
\begin{equation*}
	\Eb \langle \Eb I_n^{X} - I_n^{X}, \phi_i \rangle ^2 \leq \frac{C (C_{\rbar}, \lVert f \rVert_\infty ) }{n}
\end{equation*}
provided that $\kappa_X$ is sufficiently large.
Analogously, for the terms incorporating $I_n^\xi$ and $\Itilde_n$, we obtain with sufficiently large constants $\kappa_\xi, \kappatilde > 0$ by using Lemmata~\ref{l:E:Gxi} and \ref{l:E:Gtilde}
\begin{equation*}
	\Eb \langle \Eb I_n^{\xi} - I_n^{\xi}, \phi_i \rangle ^2 \leq \kappa_\xi \frac{\tau_n^4 (1+\log(n))}{n\privpar^4} + \frac{C(C_{\rbar}) \tau_n^4}{n^3\privpar^4}
\end{equation*}
and
\begin{equation*}
\Eb \langle \Eb \Itilde_n - \Itilde_n, \phi_i \rangle ^2 \leq \kappatilde (3 + 4\tau_n/\privpar)^4 (1+\lVert f \rVert_\infty^2) (1+ \log(n)) \frac{1}{n} + \frac{C(C_{\rbar}, \lVert f \rVert_\infty) (3+4\tau_n/\privpar)^4}{n^3},
\end{equation*}
respectively.
Putting the obtained estimates together, we get
\begin{equation}\label{eq:upper:A:2}
	\Eb \sum_{i \in \Ic_\mf} \lvert \langle \Eb I_n^{Z'} - I_n^{Z'}, \phi_i \rangle \rvert^2 \leq  D_\mf C (C_{\rbar}, \lVert f \rVert_\infty ) (1+ \log(n)) \left[\frac{1}{n} \vee \frac{\tau_n^4}{n\privpar^4} \right].
\end{equation}
Combining \eqref{eq:upper:A:1} and \eqref{eq:upper:A:2}, we obtain
\begin{equation*}
	\Eb \lVert \fhat_\mf - f_\mf \rVert^2 \1_A \lesssim D_\mf  (1+ \log(n)) \left[\frac{1}{n} \vee \frac{\tau_n^4}{n\privpar^4} \right].
\end{equation*}

\smallskip

\underline{Bound for $\Eb \lVert \fhat_\mf - f_\mf \rVert^2 \1_{\Acomplement}$}:
By the Cauchy-Schwarz inequality, we have
\begin{align}\label{EQ:DECOMP:ACOMP}
	\Eb \lVert \fhat_\mf - f_\mf \rVert^2 \1_{\Acomplement} &\leq ( \Eb \lVert \fhat_\mf - f_\mf \rVert^4 )^{1/2} \cdot ( \Pb(\Acomplement) )^{1/2},
\end{align}
and we analyse the two factors on the right-hand side separately.
First,
\begin{align*}
	\Eb \lVert \fhat_\mf - f_\mf \rVert^4 &= \Eb \left[ \left( \sum_{i \in \Ic_\mf} \lvert \langle f - \Ihat_n, \phi_i \rangle \rvert^2 \right)^2 \right]\\
	&= \Eb \left[ \left( \sum_{i \in \Ic_\mf} \lvert \langle f + \frac{8\tau_n^2}{\privpar^2} - I_n^Z, \phi_i \rangle \rvert^2 \right)^2 \right]\\
	&\leq \Eb \left[ \left( \sum_{i \in \Ic_\mf} \lVert f + \frac{8\tau_n^2}{\privpar^2} - I_n^Z \rVert^2  \right)^2 \right]\\
	&= \Eb \left[ D_\mf^2 \cdot \lVert f + \frac{8\tau_n}{\privpar^2} - I_n^Z \rVert^4 \right].
\end{align*}
Now,
\begin{align*}
	\lVert f + \frac{8\tau_n}{\privpar^2} - I_n^Z \rVert^4 &\leq 4 \pi^2 \cdot \lVert f + \frac{8\tau_n}{\privpar^2} - I_n^Z \rVert_\infty^4\\
	&\leq 32\pi^2 \cdot \lVert f + \frac{8\tau_n^2}{\privpar^2} \rVert_\infty^4 + 32 \pi^2 \lVert I_n^Z \rVert_\infty^4\\
	&\leq 32\pi^2 \left(\frac{1}{2\pi} \sum_{k \in \Z} \lvert \gamma(k) \rvert + \frac{8\tau_n^2}{\privpar^2} \right)^4 + 32\pi^2 \lVert I_n^Z \rVert_\infty^4.
\end{align*}
Furthermore, using $\lvert \Xtilde_t \rvert \leq \tau_n$,
\begin{align*}
	\Eb \left[ \lVert I_n^Z \rVert_\infty^4 \right] &= \Eb \left[ \frac{1}{(2\pi n)^4} \lVert \sum_{t=1}^n (\Xtilde_t - \Xtildebar_n) e^{-\ii t \omega} + \sum_{t=1}^n (\xi_t - \xibar_n) e^{-\ii t \omega} \rVert_\infty^8 \right]\\
	&\leq \frac{2^7}{(2\pi n)^4} \cdot \Eb \left[ \lVert \sum_{t=1}^n (\Xtilde_t - \Xtildebar_n) e^{-\ii t \omega} \rVert_\infty^8 \right] +  \frac{2^7}{(2\pi n)^4} \Eb \left[ \lVert \sum_{t=1}^n (\xi_t - \xibar_n) e^{-\ii t \omega} \rVert_\infty^8 \right]\\
	&\leq \frac{2^7}{(2\pi n)^4} \cdot \Eb \left[ \left( \sum_{t=1}^{n} \lvert \Xtilde_t - \Xtildebar_n \rvert \right)^8 \right] + \frac{2^7}{(2\pi n)^4} \cdot \Eb \left[ \left( \sum_{t=1}^n \lvert \xi_t \rvert + n \lvert \xibar_n \rvert \right)^8 \right]\\
	&\leq \frac{2^{15} n^8\tau_n^8}{(2\pi n)^4} + \frac{2^7}{(2\pi n)^4} \cdot \Eb \left[ \left( 2 \sum_{t=1}^n \lvert \xi_t \rvert \right)^8 \right] \\
	&\leq \frac{2^{11} n^4 \tau_n^8}{\pi^4} + \frac{2^{15}}{(2\pi n)^4} \cdot \Eb \left[ \left( \sum_{t=1}^n \lvert \xi_t \rvert \right)^8 \right]\\
	&= \frac{2^{11} n^4 \tau_n^8}{\pi^4} + \frac{2^{19} \tau_n^8 (n+7) \cdot (n+6) \cdot \ldots \cdot n}{\pi^4 n^4 \privpar^8}\\
	&\lesssim \frac{n^4 \tau_n^8}{1 \wedge \privpar^8}
\end{align*}
where we have also used that $\sum_{t=1}^{n} \lvert \xi_t \rvert \sim \Gamma(n,\privpar/(2\tau_n))$ together with the fact that the $k$-th moment of a $\Gamma(n,\beta)$-distributed random variable is equal to $(n + k - 1)\cdot \ldots \cdots n/\beta^k$. 
Thus,
\begin{equation}\label{EQ:fourth:moment}
	\Eb \lVert \fhat_\mf - f_\mf \rVert^4 \lesssim  D_\mf^2 \cdot \left[ \left( \sum_{k \in \Z} \lvert \gamma(k)\rvert + \frac{\tau_n^2}{\privpar^2}\right)^4 +   \frac{n^4 \tau_n^8}{1 \wedge \privpar^8} \right].
\end{equation}
Putting this bound into \eqref{EQ:DECOMP:ACOMP}, we note that it is sufficient to show that $\Pb(\Acomplement) \lesssim n^{-6}$ to obtain a bound that is bounded from above by the rate obtained for the term $\Eb \lVert \fhat_\mf - f_\mf \rVert^2 \1_{A}$ above.
We will derive such a bound in the following by means of Assumption~\ref{ASS:SUBGAUSS}.
For $n$ sufficiently large (namely $\tau_n > 2\mu$ has to hold) we have by \ref{EQ:CONC:INEQ:SUBGAUSS}
\begin{align*}
	\Pb(\Acomplement) &= \Pb( \exists i : X_i \neq \Xtilde_i)\\
	&\leq \sum_{i=1}^{n} \Pb (X_i \neq \Xtilde_i)\\
	&= \sum_{i=1}^{n} \Pb (\lvert X_i \rvert > \tau_n)\\
	&\leq \sum_{i=1}^{n} \Pb (\lvert X_i - \mu \rvert > \tau_n/2)\\
	&\leq 2\sum_{i=1}^{n} e^{-\frac{\tau_n^2}{8 \nu}}\\
	&\leq 2n e^{-\frac{\tau_n^2}{8 \nu}}.
\end{align*}

With $\tau_n^2 = 56 \nu \log(n)$ (our definition), we obtain $\Pb(\Acomplement) \lesssim n^{-6}$.
Combining this estimate with \eqref{EQ:DECOMP:ACOMP} and \eqref{EQ:fourth:moment}, we obtain desired bound for $\Eb \lVert \fhat_\mf - f_\mf \rVert^2 \1_{\Acomplement}$.
Putting the obtained bounds for the terms $\Eb \lVert \fhat_\mf - f_\mf \rVert^2 \1_{A}$ and  $\Eb \lVert \fhat_\mf - f_\mf \rVert^2 \1_{\Acomplement}$ into the right-hand side of \eqref{eq:decomp:upper} yields the claim of the theorem.

\subsection{Proof of Theorem~\ref{thm:lower} (Lower bounds)}

\subsubsection*{Proof of statement \ref{it:nonprivate:lower}}
	First, note that the minimax risk based on the sample $Z_{1:n}$ can be bounded from below by the one based on the sample $X_{1:n}$:
	\begin{align*}
		\inf_{\substack{\ftilde\\\ftilde = \ftilde (Z_{1:n})}} \sup_{f \in \Fc(\beta,L)} \Eb \lVert \ftilde - f \rVert^2 &= \inf _{\substack{\ftilde\\\ftilde = \ftilde (Q(X_{1:n}))}} \sup_{f \in \Fc(\beta,L)} \Eb \lVert \ftilde - f \rVert^2\\
		&\geq \inf _{\substack{\ftilde\\\ftilde = \ftilde (X_{1:n})}} \sup_{f \in \Fc(\beta,L)} \Eb \lVert \ftilde - f \rVert^2,
	\end{align*}
	because the original infimum on the right-hand side is taken over a smaller set of potential estimators.
	
	Put $\zeta = \min  \left\{ 1/(B \eta), 1/(2\eta), \pi/2 \right\}$.
	For any $\theta = (\theta_j)_{0 \leq j \leq \knast} \in  \{ \pm 1 \}^{\knast +1}$, we consider the function $f^\theta$ defined through
	\begin{align*}
		f^\theta &= \frac{2L}{3} + \theta_0 \left( \frac{L^2 \zeta}{9 n } \right)^{1/2} + \left( \frac{L^2 \zeta}{9 n } \right)^{1/2} \sum_{1 \leq \lvert j \rvert \leq \knast} \theta_{\lvert j \rvert} \e_j\\
		&= \frac{2L}{3} + \left( \frac{L^2 \zeta}{9 n } \right)^{1/2} \sum_{0 \leq \lvert j \rvert \leq \knast} \theta_{\lvert j \rvert} \e_j.
	\end{align*}
	Let us first check whether the functions $f^\theta$ belong to the set $\Fc(\beta, L)$ of admissible functions for any $\theta \in \{ \pm 1 \}^{\knast + 1}$.
	First, $f^ \theta$ is a real-valued function since $\ftheta_j = \ftheta_{-j}$ holds for all $j$ and all $\theta$ by construction.
	Second, $f^\theta$ is non-negative since
	\begin{align*}
		\left \lVert \left( \frac{L^2 \zeta}{9 n } \right)^{1/2} \sum_{0 \leq \lvert j \rvert \leq \knast} \theta_{\lvert j \rvert} \e_j \right \rVert_ \infty &\leq \left( \frac{L^2 \zeta}{9 n } \right)^{1/2} \sum_{0 \leq \absj \leq \knast} 1\\
		&= \left( \frac{L^2 \zeta}{9} \right)^{1/2} \left( \sum_{0 \leq \absj \leq \knast} \beta_j^{-2} \right)^{1/2} \cdot \left( \sum_{0 \leq \absj \leq \knast} \frac{\beta_j^2}{n} \right)^{1/2}\\
		&\leq \left( \frac{L^2 \zeta B}{9} \right)^{1/2}  \cdot \left( \beta_{\knast}^2 \cdot \frac{2\knast + 1}{n} \right)^{1/2}\\
		&\leq \left( \frac{L^2 \zeta B \eta}{9} \right)^{1/2}\\
		&\leq \frac{L}{3},
	\end{align*}
	and hence we even have $\ftheta \geq L/3 \geq 0$ (the fact that the functions $\ftheta$ are uniformly bounded from below will be exploited later). 
	
	Third, $\sum_{j \in \Z} \lvert f^\theta_j \rvert^2 \beta_j^2 \leq L^2$ for any $\theta \in \{ \pm 1 \}^{\knast + 1}$ thanks to the estimate
	\begin{align*}
		\sum_{j \in \Z} \lvert f^\theta_j \rvert^2 \beta_j^2  &= \sum_{0 \leq \absj \leq \knast} \lvert \ftheta_j \rvert^2 \beta_j^2\\
		&= \left[ \frac{2L}{3} + \theta_0 \left( \frac{L^2 \zeta}{9n} \right)^{1/2} \right]^2 + \frac{L^2\zeta}{9} \sum_{1 \leq \absj \leq \knast} \frac{\beta_j^2}{n}\\
		&\leq \frac{8L^2}{9} + \frac{2L^2 \zeta}{9n} + \frac{L^2 \zeta}{9} \cdot \beta_{\knast}^2 \cdot \frac{2\knast}{n}\\
		&\leq \frac{8L^2}{9} + \frac{2L^2 \zeta}{9} \cdot \beta_{\knast}^2 \cdot \frac{2\knast + 1}{n}\\
		&\leq L^2.
	\end{align*}
	Combining the three derived properties ensures $f^\theta \in \Fc(\beta,L)$.
	Denote with $\Pb_\theta$ the law of the snippet $X_{1:n}$ when $(X_t)_{t \in \Z}$ is a stationary time series with zero mean and spectral density $\ftheta$.
	Now, let $\ftilde$ be an arbitrary estimator defined in terms of the snippet $X_{1:n}$.
	Its maximal risk can be bounded from below by reduction to a finite set of hypotheses as follows:
	\begin{align}
		\sup_{f \in \Fc(\beta, L)} \Eb \lVert \ftilde - f \rVert^2 &\geq \sup_{\theta \in \{ \pm 1 \}^{\knast + 1} } \Eb_\theta \lVert \ftilde - f^\theta \rVert^2\notag \\
		&\geq \frac{1}{2^{\knast + 1}} \sum_{\theta \in \{  \pm 1 \}^{\knast +1} } \Eb_\theta \lVert \ftilde - f^\theta \rVert^2\notag\\
		&\geq \frac{1}{2^{\knast + 1}} \sum_{\theta \in \{  \pm 1 \}^{\knast +1} } \sum_{0 \leq \lvert j \rvert \leq \knast} \Eb_\theta [\lvert \ftilde_j - f^\theta_j \rvert^2]\notag\\
		&= \frac{1}{2^{\knast + 1}} \sum_{0 \leq \lvert j \rvert \leq \knast} \sum_{\theta \in \{  \pm 1 \}^{\knast +1} } \frac{1}{2} [  \Eb_\theta \lvert \ftilde_j - f^\theta_j \rvert^2 + \Eb_{\theta^{\lvert j\rvert}} \lvert \ftilde_j - f^{\theta^{\lvert j \rvert}}_j \rvert^2 ],\label{eq:red:sch}
	\end{align}
	where for $\theta \in \{ \pm 1 \}^{\knast + 1}$ and $j \in \llbracket -\knast,\knast \rrbracket$ the element $\theta^{\absj}$ is defined by $\theta^{\absj}_k = \theta_k$ for $k \neq \absj$ and $\theta^{\absj}_{\absj} = -\theta_{\absj}$ ('flip in the j-th coordinate').
	Recall the notion of Hellinger affinity which is defined via $\rho(\Pb_\theta, \Pb_{\theta^{\absj}}) = \int \sqrt{\dd \Pb_\theta \dd \Pb_{\theta^{\absj}}}$.
	For any estimator $\ftilde$, we have
	\begin{align*}
		\rho(\Pb_\theta, \Pb_{\theta^{\absj}}) &\leq  \int \frac{\lvert \ftilde_j - f^\theta_j \rvert}{\lvert f^\theta_j - f^{\theta^{\absj}}_j \rvert} \sqrt{\dd \Pb_\theta \dd \Pb_{\theta^{\absj}}} + \int \frac{\lvert \ftilde_j - f^{\theta^{\absj}}_j \rvert}{\lvert f^\theta_j - f^{\theta^{\absj}}_j \rvert} \sqrt{\dd \Pb_\theta \dd \Pb_{\theta^{\absj}}}\\
		&\leq \left( \int \frac{\lvert \ftilde_j - f^\theta_j \rvert^2}{\lvert f^\theta_j - f^{\theta^{\absj}}_j \rvert^2} \dd \Pb_\theta  \right)^{1/2}  + \left( \int \frac{\lvert \ftilde_j - f^{\theta^{\absj}}_j \rvert^2}{\lvert f^\theta_j - f^{\theta^{\absj}}_j \rvert^2} \dd \Pb_\theta^{\lvert j \rvert} \right)^{1/2},
	\end{align*}
	from which we obtain using the elementary estimate $(a+b)^2 \leq 2a^2 + 2b^2$
	\begin{equation}
	\frac{1}{2} \lvert f^\theta_j - f^{\theta^{\absj}}_j \rvert^2 \rho^2(\Pb_\theta, \Pb_{\theta^{\absj}}) \leq \Eb_\theta \lvert \ftilde_j - f^\theta_j \rvert^2+ \Eb_{\theta^{\absj}} \lvert \ftilde_j - f^{\theta^{\absj}}_j \rvert^2.\label{eq:affinity:bound}
	\end{equation}
	For the squared Hellinger distance between the laws $\Pb_\theta$ and $\Pb_{\theta^{\absj}}$ we obtain
	\begin{align*}
		H^2(\Pb_\theta,\Pb_{\theta^{\absj}}) &\leq K(\Pb_\theta,\Pb_{\theta^{\absj}})\\
		&\leq \left\lvert \Eb_ \theta \log \frac{\dd \Pb_\theta}{\dd  \Pb_{\theta^{\absj}}} \right\rvert\\
		&\leq \frac{n}{4\pi (\min_{\theta} \inf_\omega \ftheta(\omega))^2} \cdot \lVert f^\theta - f^{\theta^{\absj}} \rVert^2\\
		&\leq \frac{9n}{4\pi L^2} \cdot  \left[ \lvert f_j^\theta - f^{\theta^{\absj}}_j \rvert^2 + \lvert f_{-j}^\theta - f^{\theta^{\absj}}_{-j} \rvert^2 \right]
	\end{align*}
	by using Equation~(2.21) from \cite{tsybakov2004introduction}, Lemma~3.4 from \cite{bentkus1985rate}, and the fact that $f^\theta \geq L/3$ (the latter was \emph{en passant} established above).
	Thus, by the very definition of $\zeta$
	\begin{align*}
		H^2(\Pb_\theta,\Pb_{\theta^{\absj}}) &\leq \frac{18n}{\pi L^2} \cdot  \left( \frac{L^2 \zeta}{9n} \right) \leq 1,
	\end{align*}
	and consequently $\rho(\Pb_\theta, \Pb_{\theta^{\absj}}) \geq 1/2$.
	Putting this estimate into \eqref{eq:affinity:bound} and combining it with \eqref{eq:red:sch} yields
	\begin{align*}
		\sup_{f \in \Fc(\beta, L)} \Eb \lVert \ftilde - f \rVert^2 &\geq \frac{1}{4} \sum_{0 \leq \absj \leq \knast} \frac{L^2 \zeta}{9n} = \frac{L^2 \zeta}{36} \cdot \frac{2\knast + 1}{n}
	\end{align*}
	which is the claim assertion.
	
	\subsubsection*{Proof of statement \ref{it:private:lower}}
	Set 
	\begin{equation*}
		\Psinpp^2 = \frac{1}{4} \min \left\{  \frac{\pi}{n(e^\privpar - 1)^2} , \frac{L^2}{4}  \right\}.
	\end{equation*}
	Grant to the general reduction principle for the proof of minimax lower bounds (see Chapitre 2.2 in \cite{tsybakov2004introduction}) it is suffcient to find two candidate functions $f^{0},f^{1}$ such that
	\begin{enumerate}[(i)]
		\item $f^0, f^1 \in \Fc(\beta,L)$,\label{it:check:lower:i}
		\item $\lVert f^0 - f^1 \rVert_2^2 \gtrsim 4 \Psinpp^2$, and\label{it:check:lower:ii}
		\item $\KL(\Pb_f^Z, \Pb_g^Z) \leq C$ for some constant $C < \infty$ depending neither on $\privpar$ nor $n$.\label{it:check:lower:iii}
	\end{enumerate}
	Then, for any estimator $\ftilde$
	\begin{align*}
		\sup_{f \in \Fc(\beta,L)} \Eb \lVert \ftilde - f \rVert_2^2 &\geq \Psinpp^2 \sup_{f \in \Fc(\beta,L)} \Pb(\lVert \ftilde - f \rVert \geq \Psinpp)\\
		&\geq \Psinpp^2 \sup_{\theta \in \{ 0,1 \}} \Pb(\lVert \ftilde - f^\theta \rVert \geq \Psinpp)\\
		&\geq \Psinpp^2 \inf_{T} \max_{\theta \in \{ 0,1 \} } \Pb_\theta(T \neq \theta)
	\end{align*}
	where the last infimum runs over all tests $T$ with values in $\{ 0,1 \}$ and $\Pb_\theta$ denotes the distribution of $Z$ when the true spectral density is $\ftheta$.
	
	Let us define the functions $\ftheta$ for $\theta \in \{ 0,1 \}$ via
	\begin{align*}
		f^0 &\equiv L,\\
		f^1 &\equiv f^0 - \min \left\{ L - \sqrt{\frac{\pi}{n(e^\privpar - 1)^2}} ,L/2 \right\} = L - \min \left\{\sqrt{\frac{\pi}{n(e^\privpar - 1)^2}} ,L/2 \right\},
	\end{align*}
	and we need to verify the conditions \ref{it:check:lower:i}--\ref{it:check:lower:iii}.
	Condition \ref{it:check:lower:i} is trivially satisfied and \ref{it:check:lower:ii} follows from the identity
	\begin{equation*}
		\lVert f^0 - f^1 \rVert_2^2 = (f^0_0 - f^1_0)^2 = \min \left\{  \frac{\pi}{n(e^\alpha - 1)^2} , \frac{L^2}{4}  \right\} = 4\Psinpp^2.
	\end{equation*}
	It remains to prove \ref{it:check:lower:iii}.
	First note that the fact that both candidate spectral densites $\ftheta$ are constant ensures, by Gaussianity, that the random variables $X_1,\ldots,X_n$ are independent.
	Thus, we can apply Corollary~1 from \cite{duchi2018minimax} together with Lemma~3.4 from \cite{bentkus1985rate} and the bound $\TV^2 \leq \KL$ (see (2.21) in \cite{tsybakov2004introduction}, for instance) in order to obtain
	\begin{align*}
		\KL(\Pb^{Z}_0, \Pb^Z_1) &\leq 4 (e^\privpar - 1)^2 \sum_{i=1}^{n} \TV^2(\Pb_0^{X_i}, \Pb_1^{X_i})\\
		&\leq  4(e^\privpar - 1)^2 \sum_{i=1}^{n} \KL(\Pb_0^{X_i}, \Pb_1^{X_i}) \\
		&= 4(e^\privpar - 1)^2 \KL(\Pb_0^{X}, \Pb_1^{X})\\
		&\leq \frac{(e^\privpar - 1)^2n}{\pi (\min_ {\theta = 0,1} \inf_\omega \ftheta(\omega))^2} \cdot \lVert f^0 - f^1 \rVert^2_2\\
		&= \frac{4(e^\privpar - 1)^2n}{\pi L^2} \cdot (f_0 - g_0)^2\\
		&= \frac{4(e^\privpar - 1)^2n}{\pi L^2} \cdot \min \left\{\sqrt{\frac{\pi}{n(e^\privpar - 1)^2}} ,L/2 \right\}^2\\
		&\leq \frac{4(e^\privpar - 1)^2n}{\pi L^2} \cdot \frac{\pi}{n(e^\privpar - 1)^2}\\
		&= 4/L^2.
	\end{align*}
	Now, application of  Th{\' e}or{\`e}me 2.2., (iii) from \cite{tsybakov2004introduction} yields the bound
	\begin{equation*}
		\inf_{T} \max_{\theta \in \{ 0,1 \} } \Pb_\theta(T \neq \theta) \geq \max \left\{ \frac{1}{4} e^{-4/L^2}, \frac{1 - \sqrt{2}/L}{2} \right\}
	\end{equation*}
	which finishes the proof. 
\section{Proofs of Section~\ref{s:adaptation}}
	We define the event $A$ (and its complement) exactly as in the proof of Theorem~\ref{thm:upper:minimax}, namely
	\begin{equation*}
		A = \bigcap_{i=1}^{n} \{ X_i = \Xtilde_i \},
	\end{equation*}
	and consider the decomposition
	\begin{equation*}
		\Eb \lVert \ftilde - f \rVert^2 = \Eb \lVert \ftilde - f \rVert^2 \1_A + \Eb \lVert \ftilde - f \rVert^2 \1_{A^\complement}.
	\end{equation*}
	
	\smallskip
	
	\noindent \underline{Upper bound for $\Eb \lVert \ftilde - f \rVert^2 \1_A$}: We can write the contrast as
	\begin{equation*}
		\Upsilon_n(t) = \lVert t \rVert^2 - 2 \langle \Ihat_n, t \rangle = \lVert t-f \rVert^2 - 2 \langle \Ihat_n - f,t \rangle - \lVert f \rVert^2.
	\end{equation*}
	By the definitions of $\ftilde$ and $\mfhat$combined with the fact that $\fhat_{\mf}$ minimizes the contrast over the space $S_\mf$ the estimate
	\begin{equation*}
		\Upsilon_n(\ftilde) + \pen(\mfhat) \leq \Upsilon_n(f_\mf) + \pen(\mf)
	\end{equation*}
	holds for all $\mf \in \Mc_n$, we obtain
	\begin{equation*}
		\lVert f - \ftilde \rVert^2 - 2 \langle \Ihat_n - f, \ftilde \rangle + \pen(\mfhat) \leq \lVert f - f_\mf \rVert^2 - 2 \langle \Ihat_n - f, f_\mf \rangle + \pen(\mf).
	\end{equation*}
	Then, by elementary algebraic manipulations,
	\begin{align*}
		\lVert f - \ftilde \rVert^2 &\leq \lVert f - f_\mf \rVert^2 + 2 \langle \Ihat_n - f, \ftilde - f_\mf \rangle + \pen(\mf) - \pen(\mfhat)\\
		&\leq \lVert f - f_\mf \rVert^2 + 2 \langle f - \Eb \Ihat_n, f_\mf - \ftilde \rangle + 2 \langle \Ihat_n - \Eb \Ihat_n, \ftilde - f_\mf \rangle + \pen(\mf) - \pen(\mfhat).
	\end{align*}
	On the event $A$, we have $Z'_{1:n}=Z_{1:n}$ and $\Ihat_n = I^Z_n - \frac{8\tau_n^2}{\privpar^2} = I^{Z'}_n - \frac{8\tau_n^2}{\privpar^2}$.
	Hence, on $A$ the identity
	\begin{equation*}
		\langle \Ihat_n - \Eb \Ihat_n, \ftilde - f_\mf \rangle = \langle I_n^{Z'} - \Eb I_n^{Z'}, \ftilde - f_\mf \rangle 
	\end{equation*}
	holds.
	By definition of $I_n^{Z'}$, we have
	\begin{align*}
		I_n^{Z'}(\omega) &= \frac{1}{2\pi n} \left\lvert \sum_{t=1}^n (Z_t^\prime - \Zprimebarn) e^{-\ii t \omega} \right\rvert^2\\
		&= \frac{1}{2\pi n} \left\lvert \sum_{t=1}^n (X_t - \Xbarn) e^{-\ii t \omega} + \sum_{t=1}^n (\xi_t - \xibarn) e^{-\ii t \omega} \right\rvert^2\\
		&= \frac{1}{2\pi n} \left\lvert \sum_{t=1}^n (X_t - \Xbarn) e^{-\ii t \omega} \right\rvert^2 + \frac{1}{2\pi n} \left( \sum_{t=1}^n (X_t - \Xbarn) e^{-\ii t \omega} \right) \left( \sum_{t=1}^n (\xi_t - \xibarn) e^{\ii t \omega} \right) \\
		&\hspace{1em}+ \frac{1}{2\pi n} \left( \sum_{t=1}^n (X_t - \Xbarn) e^{\ii t \omega} \right) \left( \sum_{t=1}^n (\xi_t - \xibarn) e^{-\ii t \omega} \right) + \frac{1}{2\pi n}\left\lvert \sum_{t=1}^n (\xi_t - \xibarn) e^{-\ii t \omega} \right\rvert^2\\
		&\eqdef I_n^X + \Itilde_n + I_n^\xi
	\end{align*}
	(as above $\Itilde_n$ is defined as the sum of the two 'mixed' terms).
	For $\mf, \mf' \in \Mc_n$, set
	\begin{align*}
		G_X(\mf,\mfprime) &= \sup_{u \in \Bc_{\mf, \mfprime}} \langle I_n^X - \Eb I_n^X, u \rangle,\\
		G_\xi(\mf,\mfprime) &= \sup_{u \in \Bc_{\mf, \mfprime}} \langle I_n^\xi - \Eb I_n^\xi, u \rangle,\\
		\Gtilde(\mf,\mfprime) &= \sup_{u \in \Bc_{\mf, \mfprime}} \langle \Itilde_n, u \rangle,
	\end{align*}
	where $\Bc_{\mf, \mfprime}$ denotes the unit ball in $S_\mf + S_{\mfhat}$, and we write $G_X(\mf)$, $G_\xi(\mf)$, and $\Gtilde(\mf)$ when $\mf = \mfprime$.
	We have $G_X(\mf,\mfprime) \leq G_X(\mf) + G_X(\mfprime)$, and the same type of bound holds for $G_\xi$ and $\Gtilde$.
	As a consequence, using the estimate $2xy \leq \tau x^2 + \tau^{-1} y^2$ for $\tau = 16$ we have
	\begin{align*}
		\lVert f - \ftilde \rVert^2 \1_A &\leq  \left( \lVert f - f_\mf \rVert^2 + 2 \langle f - \Eb I^{Z'}_n, f_\mf - \ftilde \rangle + 2 \langle I^{Z'}_n - \Eb I^{Z'}_n, \ftilde - f_\mf \rangle + \pen(\mf) - \pen(\mfhat) \right) \1_A \\
		&= \left(  \lVert f - f_\mf \rVert^2 + 2 \langle f - \Eb I^{Z'}_n, f_\mf - \ftilde \rangle + 2 \langle I_n^X - \Eb I_n^X, \ftilde - f_\mf \rangle +2 \langle I_n^\xi - \Eb I_n^\xi, \ftilde - f_\mf \rangle \right.\\
		&\hspace{1em} \left. + 2 \langle \Itilde_n, \ftilde - f_\mf \rangle + \pen(\mf) - \pen(\mfhat) \right) \1_A \\
		&\leq \left( \lVert f - f_\mf \rVert^2 + 2 \langle f - \Eb I^{Z'}_n, f_\mf - \ftilde \rangle + 2 \lVert \ftilde - f_\mf \rVert G_X(\mf, \mfhat) + 2 \lVert \ftilde - f_\mf \rVert G_\xi(\mf,\mfhat) \right.\\
		&\hspace{1em}\left. + 2\lVert \ftilde - f_\mf \rVert \Gtilde(\mf,\mfhat) + \pen(\mf) - \pen(\mfhat) \right) \1_A \\
		&= ( \lVert f - f_\mf \rVert^2 + \tau \lVert f - \Eb I^{Z'}_n \rVert^2 + 4\tau^{-1} \lVert f_\mf - \ftilde \rVert^2 + \tau G_X^2(\mf,\mfhat) + \tau G_\xi^2(\mf,\mfhat) \\
		&\hspace{1em} + \tau \Gtilde^2(\mf,\mfhat)  + \pen(\mf) - \pen(\mfhat) )  \1_A \\
		&=( \lVert f - f_\mf \rVert^2 + 16 \lVert f - \Eb I^{Z'}_n \rVert^2 + \frac 1 4 \lVert f_\mf - \ftilde \rVert^2 + 32 G_X^2(\mfhat) + 32 G_\xi^2(\mfhat) \\
		&\hspace{1em}+ 32 \Gtilde^2(\mfhat) + 32 G_X^2(\mf) + 32 G_\xi^2(\mf) +  32 \Gtilde^2(\mf)  + \pen(\mf) - \pen(\mfhat) ) \1_A\\
		&\leq (  3 \lVert f - f_\mf \rVert^2/2 + 16 \lVert f - \Eb I^{Z'}_n \rVert^2 + \frac{1}{2} \Eb \lVert f - \ftilde \rVert^2 + 32 G_X^2(\mfhat) + 32 G_\xi^2(\mfhat) \\
		&\hspace{1em} + 32 \Gtilde^2(\mfhat) + 32 G_X^2(\mf) + 32 G_\xi^2(\mf) +  32 \Gtilde^2(\mf) + \pen(\mf) - \pen(\mfhat) ) \1_A.
	\end{align*}
	Hence,
	\begin{align*}
		\lVert f - \ftilde \rVert^2 \1_A &\leq (  3 \lVert f - f_\mf \rVert^2 + 32 \lVert f - \Eb I^{Z'}_n \rVert^2 + 64 G_X^2(\mfhat) + 64 G_\xi^2(\mfhat) \\
		& + 64 \Gtilde^2(\mfhat) + 64 G_X^2(\mf) + 64 G_\xi^2(\mf) +  64 \Gtilde^2(\mf)  + 2\pen(\mf) - 2\pen(\mfhat) )  \1_A.
	\end{align*}
	If the numerical constant in the definition of the penalty is large enought, we can write $\pen(\mf) = \pen_X(\mf) + \pen_\xi(\mf) + \pentilde(\mf)$ such that
	\begin{align*}
		\pen_X(\mf) &\geq 32\kappa_X \lVert f \rVert_\infty (1+ C_{\rbar}^2) \frac{D_{\mf} ( 1 + L_{\mf})^2}{n},\\
		\pen_\xi(\mf) &\geq 32 \kappa_\xi \frac{\tau_n^4D_{\mf} (L_{\mf}^4 + L_\mf + \log(n) ) }{n\privpar^4}, \qquad \text{and}\\
		\pentilde(\mf) &\geq 32 \kappatilde M^4 (1+\lVert f \rVert_\infty)^2 (L_{\mf}^4 + L_\mf + \log(n) ) \frac{D_{\mf}}{n}
	\end{align*}
	holds for any model $\mf \in \Mc_n$.
	Summing over all potential models and taking expectations implies
	\begin{align*}
		\Eb \lVert f - \ftilde \rVert^2 \1_A &\leq 3 \lVert f - f_\mf \rVert^2 + 32 \lVert f - \Eb I^{Z'}_n \rVert^2  + 4 \pen(\mf) \\
		&+ 128 \sum_{\mf \in \Mc_n} \Eb \left[ \left( G_X^2(\mf) - \pen_X(\mf)/32 \right)_+  \right] \\
		&+ 128 \sum_{\mf \in \Mc_n} \Eb \left[ \left( G_\xi^2(\mf) - \pen_\xi(\mf)/32 \right)_+  \right]\\
		&+ 128 \sum_{\mf \in \Mc_n} \Eb \left[ \left( \Gtilde^2(\mf) - \pentilde(\mf)/32 \right)_+  \right].
	\end{align*}
	The expectations are bounded by Lemmata~\ref{l:E:GX}, \ref{l:E:Gxi}, and \ref{l:E:Gtilde}, combined with Assumption~\ref{ass:weights} in order to obtain
	\begin{align*}
		\Eb \lVert f - \ftilde \rVert^2 \1_A &\leq 3 \lVert f - f_\mf \rVert^2 + 32 \lVert f - \Eb I^{Z'}_n \rVert^2  + 4\pen(\mf) \\
		&+ C(C_{\rbar},\lVert f \rVert_\infty) \max \left\lbrace \frac{1}{n}, \frac{\tau_n^4}{n^3 \privpar^4} \right\rbrace.
	\end{align*}
	Finally, by Proposition~\ref{prop:ex:comte} we get (using the same argument as in the proof of Theorem~\ref{thm:upper:minimax})
		\begin{align*}
			\Eb \lVert f - \ftilde \rVert^2 \1_A &\lesssim \lVert f - f_\mf \rVert^2 + \max \left\lbrace  \frac{\tau_n^4}{n \privpar^4} , \frac{1}{n} \right\rbrace  + \pen(\mf) \\
			&\hspace{1em}+ C(C_{\rbar},\lVert f \rVert_\infty) \max \left\lbrace \frac{1}{n}, \frac{\tau_n^4}{n^3 \privpar^4} \right\rbrace.
		\end{align*}
	Since, this estimate holds for any fixed model $\mf$, we can take the infimum over all potential models which yields
	\begin{align*}
		\Eb \lVert f - \ftilde \rVert^2 \1_A &\lesssim \inf_{\mf \in \Mc_n} \left[ \lVert f - f_\mf \rVert^2 , \pen(\mf) \right]  + \max \left\lbrace  \frac{\tau_n^4}{n \privpar^4} , \frac{1}{n} \right\rbrace \\
		&\hspace{1em}+ C(C_{\rbar},\lVert f \rVert_\infty) \max \left\lbrace \frac{1}{n}, \frac{\tau_n^4}{n^3 \privpar^4} \right\rbrace.
	\end{align*}
	
	\noindent \underline{Upper bound for $\Eb \lVert \ftilde - f \rVert^2 \1_{A^\complement}$:}
	This term can be bounded exactly as in the upper bound for any fixed model (the only property of the model that we have exploited in that proof was the fact that $D_{\mf} \leq n$ which holds true also for the randomly selected model $\mfhat$):
\begin{equation*}
	\Eb \lVert \fhat_\mf - f_\mf \rVert^2 \1_{\Acomplement} \lesssim \frac{1}{n}.
\end{equation*} 
\section{Concentration results for the proof of Theorem~\ref{thm:adaptation}}

\subsection{A general chaining argument}

Let $\Sbar$ be a finite dimensional subspace of $L^2 \cap L^\infty$ spanned by some orthonormal basis $(\phi_i)_{i \in \Ic}$.
We denote the dimension $\lvert \Ic \rvert$ of $\Sbar$ with $D$, and define the quantity
\begin{equation*}
	\rbar_\phi = \frac{1}{\sqrt D} \sup_{\beta \in \R^D, \beta \neq 0} \frac{\lVert \sum_{i \in \Ic} \beta_i \phi_i \rVert_\infty}{\lvert \beta \rvert_\infty}.
\end{equation*}
In addition, we define $\rbar$ as the infimum of $\rbar_\phi$ taken over all possible orthonormal bases of $\Sbar$.

\begin{proposition}[Proposition~1 from \cite{birge1998minimum}]\label{prop:bm98}
	Let $\Sbar$ be a $D$-dimensional linear subspace of $L^2 \cap L^\infty$ with its index $\rbar$ defined as above.
	Let $\Bc$ be any ball of radius $\sigma$ in $\Sbar$ and $0 < \delta < \sigma/5$.
	Then there exists a finite set $T \subset \Bc$ which is simultaneously a $\delta$-net for $\Bc$ with respect to the $L^2$-norm and an $\rbar \delta$-net with respect to the $L^\infty$-norm and such that $\lvert T \rvert \leq (6\sigma/\delta)^D$.
\end{proposition}

We will apply Proposition~\ref{prop:bm98} with $\sigma = 1$ which reduces the choice of $\delta$ to $\delta < \frac{1}{5}$.

In the sequel, we will use the following chaining argument.
For $0< \delta_0 < 1/5$ and any $k \in \N$, we set $\delta_k = 2^{-k} \delta_0$ and consider a sequence of $\delta_k$-nets $(T_k)_{k \in \N}$ with $T_k = T_{\delta_k}$.
Then, for any $u \in \Bc_{\mf}$ ($\Bc_\mf$ is defined in the proof of Theorem~\ref{thm:adaptation} as the unit ball in the space $S_\mf$), we are able to find a sequence $(u_k)_{k \geq 0}$ with $u_k \in T_k$ such that $\lVert u - u_k \rVert^2 \leq \delta_k^2$ and $\lVert u-u_k \rVert_\infty \leq \rbar_{\mf} \delta_k$. 
Moreover, one can achieve $\lvert T_k \rvert \leq (6/\delta_k)^{D_\mf}$.
We have the following decomposition:
\begin{equation}\label{eq:dec:u}
	u = u_0 + \sum_{k=1}^\infty (u_k - u_{k-1}).
\end{equation}
From the above properties it follows that $\lVert u_0 \rVert \leq \delta_0$, $\lVert u_0 \rVert_\infty \leq \rbar_{\mf} \delta_0$, and, for $k \geq 1$, $\lVert u_k - u_{k-1} \rVert^2 \leq 2(\delta_k^2 + \delta_{k-1}^2) = 5\delta_{k-1}^2/2$ and $\lVert u_k - u_{k-1} \rVert_\infty \leq 3 \rbar_{\mf} \delta_{k-1}/2$.
These estimates will be used below without further reference.

Let us finally note that we will work with different definitions of $\delta_0$ below.
For the purely Gaussian terms in Subsection \ref{subs:Gauss} it will turn out convenient to choose $0 < \delta_0  < 1/5$ as a numerical constant independent of $n$ whereas for the analysis of the Laplace term in Subsection \ref{subs:Laplace} and the mixed term \ref{subs:mixed} we will need to choose $\delta_0 \asymp n^{-1}$ in order to get better rates (at the cost of slightly worse logarithmic terms).
We put $H_k = \log(\lvert T_k \rvert)$.
Then
\begin{equation*}
	H_k \leq D_{\mf} \log (6/\delta_k) = D_{\mf} [\log(6/\delta_0) + k \log 2]
\end{equation*}
which will be used below without further reference.

\subsection{The Toeplitz matrix $T_n(u)$}

In the following three Subsections~\ref{subs:Gauss}--\ref{subs:mixed} we will consider the following Toeplitz matrix $T_n(u)$ associated with the function $u$ that is given by the entries
\begin{equation*}
[T_n(u)]_{j,k} = \int_{-\pi}^\pi u(\omega) \ee^{\ii \omega (j-k)} \dd \omega, \qquad 1 \leq j,k \leq n.
\end{equation*}
The matrix $T_n(u)$ is always Hermitian but since we consider only symmetric $u$, the same holds true for $T_n(u)$ (which is then real-valued).

\subsection{Gaussian terms}\label{subs:Gauss}

\begin{proposition}\label{prop:prob:Xi:X}
	Put $\Xi_n^X(u) = \langle I_n^X - \Eb I_n^X, u \rangle$. For any symmetric function $u$,
	\begin{equation*}
	\Pb(\Xi_n^X(u) \geq t) \leq 2 \exp \left[  - c \min \left( \frac{4\pi^2nt^2}{9 \lVert f\rVert_\infty^2 \lVert u\rVert^2}, \frac{2\pi nt}{3 \lVert f \rVert_\infty \lVert u \rVert_\infty} \right) \right] .
	\end{equation*}
\end{proposition}

\begin{proof}
Denote $X=(X_1,\ldots,X_n)^\top$. First, we can write
\begin{equation*}
\Xi_n^X(u) = \frac{1}{2\pi n} [ (X - \Xbar_n \vec 1)^\top T_n(u) (X - \Xbar_n \vec 1) - \Eb (X - \Xbar_n \vec 1)^\top T_n(u) (X - \Xbar_n \vec 1) ].
\end{equation*}
Let $H$ be the hyperplane orthogonal to the linear subspace generated by the vector $\vec 1$ in $\R^n$.
Note that $X- \Xbar_n \vec 1 = P_H X = P_H \Sigma_X^{1/2} Y$ where $Y \sim \Nc(\vec 0, E_n)$ and $\Sigma_X$ is the covariance matrix of $X_{1:n}$.
Now, we the Hanson-Wright inequality (Proposition~\ref{HWI}) with $A = (\Sigma_X^{1/2})^\top P_H^\top T_n(u) P_H \Sigma_X^{1/2}$.
Since the $Y_i$ are i.i.d. $\sim \Nc(0,1)$, we have $\lVert Y_i \rVert_{\psi_2} \leq \sqrt{8/3} \leq \sqrt{3} = K$.
For the given choice of $A$, we need to bound the quantities $\lVert A \rVert_{\mathrm{HS}}$ and $\lVert A \rVert_{\mathrm{op}}$ appearing on the right-hand side of the Hanson-Wright inequality.
First,
\begin{align*}
\lVert A \rVert_{\mathrm{HS}}^2 &= \trace(A^\top A) =\trace(\Sigma_X^{1/2} P_H^\top T_n(u) P_H \Sigma_X P_H^\top T_n(u) P_H \Sigma_X^{1/2})\\
&= \trace( P_H \Sigma_X P_H^\top T_n(u) P_H \Sigma_X P_H^\top T_n(u))\\
&\leq \lVert f \rVert_\infty^2 \cdot \trace( T_n(u)^2)\\
&\leq n \lVert f \rVert_\infty^2 \lVert u \rVert^2,
\end{align*}
where we have used the bound $\trace((AB)^2) \leq \rho(A)^2 \trace(B^2)$, and the fact that $\trace(T_n(u)^2) \leq n \lVert u \rVert^2$ from p.~284 in \cite{comte2001adaptive}. Second,
\begin{align*}
\lVert A \rVert_{\mathrm{op}} &= \lVert \Sigma_X^{1/2} P_H^\top T_n(u) P_H \Sigma_X^{1/2} \rVert_{\mathrm{op}}\\
&\leq \lVert \Sigma_X^{1/2} \rVert_{\mathrm{op}} \cdot \lVert T_n(u) \rVert_{\mathrm{op}} \cdot \lVert \Sigma_X^{1/2} \rVert_{\mathrm{op}}\\
&= \lVert \Sigma_X \rVert_{\mathrm{op}} \cdot \lVert T_n(u) \rVert_{\mathrm{op}}\\
&= \rho(\Sigma_X) \cdot \rho(T_n(u))\\
&\leq \lVert f \rVert_\infty \cdot \lVert u \rVert_\infty.
\end{align*}
Using these estimates, application of the Hanson-Wright inequality (Proposition~\ref{HWI}) yields
\begin{equation*}
\Pb ( \Xi^X(u) \geq t ) \leq 2 \exp \left[  - c \min \left( \frac{4\pi^2nt^2}{9 \lVert f\rVert_\infty^2 \lVert u\rVert^2}, \frac{2\pi nt}{3\lVert f \rVert_\infty \lVert u \rVert_\infty} \right) \right].
\end{equation*}
\end{proof}

\begin{lemma}\label{l:E:GX} For any fixed model $\mf \in \Mc_n$ and a sufficiently large constant $\kappa_X > 0$, we have
	\begin{equation*}
		\Eb \left[ \left( (G^X(\mf))^2 - \kappa_X \lVert f \rVert_\infty^2  (1 + C_{\rbar}^2)\frac{D_{\mf} (1 + L_{\mf})^2}{n}  \right)_+ \right]  \lesssim e^{- L_{\mf} D_{\mf}} \cdot \frac{C (C_{\rbar}, \lVert f \rVert_\infty ) }{n}.
	\end{equation*}
\end{lemma}

\begin{proof}
	We consider a sequence $(\eta_k)_{k \geq 0}$ of positive numbers and $\eta \geq \sum_{k \geq 0}\eta_k$ (these quantities will be specified later on).
	Then, using the decomposition \eqref{eq:dec:u},
	\begin{align*}
		\Pb(\sup_{u \in B_{\mf}} \Xi^X_n(u) > \eta) &=\Pb \left[ \exists (u_k)_{k \geq 0} \in \prod_{k \geq 0} T_k : \Xi^X(u_0) + \sum_{k \geq 1} \Xi_n^X(u_k - u_{k-1}) > \eta_0 + \sum_{k \geq 1} \eta_k \right] \\
		&\leq P_1 + P_2,
	\end{align*}
	where
	\begin{align*}
		P_1 &= \sum_{u_0 \in T_0} \Pb(\Xi^X_n(u_0) > \eta_0),\\
		P_2 &= \sum_{k\geq 1} \sum_{\substack{u_{k-1} \in T_{k-1}\\u_k \in T_k}} \Pb ( \Xi_n^X(u_k - u_{k-1}) > \eta_k ).
	\end{align*}
	For any $u_0 \in T_0$, we obtain from Proposition~\ref{prop:prob:Xi:X} that
	\begin{align*}
		\Pb(\Xi^X_n(u_0) > \eta_0) &\leq 2 \exp \left( - c  \min \left( \frac{4\pi^2 n\eta_0^2}{9 \lVert f \rVert_\infty^2 \delta_0^2} , \frac{2\pi n\eta_0}{3 \lVert f\rVert_\infty \rbar_{\mf} \delta_0} \right)  \right),
	\end{align*}
	and hence
	\begin{equation*}
		P_1 \leq 2 \exp(H_0) \exp \left( - c  \min \left( \frac{4\pi^2 n\eta_0^2}{9 \lVert f \rVert_\infty^2 \delta_0^2} , \frac{2\pi n\eta_0}{3 \lVert f\rVert_\infty \rbar_{\mf} \delta_0} \right) \right).
	\end{equation*}
	For $\lambda > 0$, we consider $\eta_0$ such that
	\begin{equation*}
		c  \min \left( \frac{n\eta_0^2}{9 \lVert f \rVert_\infty^2 \delta_0^2} , \frac{n\eta_0}{3 \lVert f\rVert_\infty \rbar_{\mf} \delta_0} \right) \geq H_0 + L_{\mf}D_{\mf} + \lambda,
	\end{equation*}
	that is,
	\begin{equation*}
		\eta_0 = C \lVert f \rVert_\infty \delta_0 \cdot \max \left(\sqrt{\frac{H_0 + L_{\mf}D_{\mf} + \lambda}{n}}, \frac{\rbar_{\mf} (H_0 + L_{\mf}D_{\mf} + \lambda)}{n} \right).
	\end{equation*}
	for some sufficiently large constant $C > 0$.
For any $k \geq 1$, we get from Proposition~\ref{prop:prob:Xi:X} with $u_{k-1} \in T_{k-1}$ and $u_k \in T_k$
\begin{align*}
	\Pb ( \Xi_n^X(u_k - u_{k-1}) > \eta_k) &\leq 2 \exp \left( - c\min \left( \frac{8\pi^2 n\eta_k^2}{45 \lVert f\rVert_\infty^2 \delta_{k-1}^2}, \frac{4\pi n\eta_k}{9\lVert f\rVert_\infty \rbar_{\mf} \delta_{k-1}} \right)  \right).
\end{align*}
Here, for $\lambda \geq 0$, we choose the $\eta_k$ such that
\begin{equation*}
	c\min \left( \frac{8 \pi^2 n\eta_k^2}{45 \lVert f\rVert_\infty^2 \delta_{k-1}^2}, \frac{4\pi n\eta_k}{9\lVert f\rVert_\infty \rbar_{\mf} \delta_{k-1}} \right)  \geq H_{k-1} + H_k + kD_{\mf} + L_{\mf} D_{\mf} + \lambda
\end{equation*}
which in turn is satisfied whenever
\begin{equation*}
	\eta_k = C \lVert f \rVert_\infty \delta_{k-1} \max \left( \sqrt{\frac{H_{k-1} + H_k + k D_{\mf} + L_{\mf}D_{\mf} + \lambda}{n}}, \frac{\rbar_{\mf} (H_{k-1} + H_k + k D_{\mf} + L_{\mf}D_{\mf} + \lambda)}{n} \right)
\end{equation*}
for some sufficiently large constant $C>0$.
Under this choice of $(\eta_k)_{k \geq 0}$, we obtain for $\eta \geq \sum \eta_k$ (using the assumption that $D_{\mf} \geq 1$)
\begin{align*}
	\Pb(\sup_{u \in B_{\mf}} \Xi^X_n(u) > \eta) &\leq P_1 + P_2\\
	&\leq 2 \exp(-L_{\mf}D_{\mf}- \lambda) + 2 \sum_{k \geq 1} \exp(-kD_{\mf}- L_{\mf}D_{\mf} - \lambda)\\
	&= 2 \exp (- L_{\mf}D_{\mf} - \lambda) \left[ 1 + \sum_{k\geq 1} e^{-kD_{\mf}} \right]\\
	&\leq 3.2 \exp (- L_{\mf}D_{\mf} - \lambda).
\end{align*}
We compute a bound for $\sum_{k \geq 0} \eta_k$, and take $0< \delta_0<1/5$ as a purely numerical constant from now on.\small
\begin{align*}
	\left( \sum_{k \geq 0} \eta_k \right)^2 &\leq C \lVert f \rVert_\infty^2 \left( \delta_0 \sqrt{\frac{H_0 + L_{\mf}D_{\mf} + \lambda}{n}} + \sum_{k \geq 1} \delta_{k-1} \sqrt{\frac{H_{k-1} + H_k + kD_{\mf} + L_{\mf}D_{\mf} + \lambda}{n}} \right.\\
	&\left.+ \delta_0 \frac{\rbar_{\mf} (H_0  + L_{\mf}D_{\mf} + \lambda)}{n} + \sum_{k \geq 1} \delta_{k-1} \frac{\rbar_{\mf} (H_{k-1} + H_k + kD_{\mf} + L_{\mf}D_{\mf} + \lambda)}{n} \right)^2\\
	&\leq C \lVert f \rVert_\infty^2 \left( \frac{1}{n}(\delta_0 + \sum_{k \geq 1} \delta_{k-1}) (\delta_0 (H_0 + L_{\mf}D_{\mf} + \lambda) + \sum_{k \geq 1} \delta_{k-1} (H_{k-1} + H_k + kD_{\mf} + L_{\mf}D_{\mf} + \lambda)) \right.\\
	&\left. + \frac{\rbar^2_{\mf}}{n^2} \left( \delta_0(H_0 + L_{\mf}D_{\mf} + \lambda) + \sum_{k \geq 1}\delta_{k-1} H_{k-1} + H_k + kD_{\mf} + L_{\mf}D_{\mf} + \lambda \right)^2 \right)\\
	&\leq C \lVert f \rVert_\infty^2 \left[ \left(  \frac{D_{\mf} + D_{\mf} L_{\mf} + \lambda}{n} \right) + \frac{\rbar^2_{\mf}}{n^2} (D_{\mf}^2 + D_{\mf}^2L_{\mf}^2 + \lambda^2) \right]  \\
	&\leq C \lVert f \rVert_\infty^2 \left[  \frac{D_{\mf}(1 + L_{\mf})}{n} + \frac{\lambda}{n} + \frac{C_{\rbar}^2 D_{\mf}(1+ L_{\mf}^2)}{n} +  \frac{\rbar^2_{\mf} \lambda^2}{n^2} \right]\\
	&\leq \kappa_X \lVert f \rVert_\infty^2  (1 + C_{\rbar}^2)\frac{D_{\mf} (1 + L_{\mf})^2}{n} + 2 \left[ \frac{\lambda}{n} \vee \frac{\rbar^2_{\mf} \lambda^2}{n^2} \right] 
\end{align*}
for some numerical constant $\kappa_X$.
Then,
\begin{align*}
  \Eb &\left[ \left( (G^X(\mf))^2 - \kappa_X \lVert f \rVert_\infty^2  (1 + C_{\rbar}^2)\frac{D_{\mf} (1 + L_{\mf})^2}{n}  \right)_+ \right]\\
  &= \int_0^\infty \Pb \left( (G^X(\mf))^2 > \kappa_X \lVert f \rVert_\infty^2  (1 + C_{\rbar}^2)\frac{D_{\mf} (1 + L_{\mf})^2}{n} + u \right) \dd u\\
  &\leq e^{-L_{\mf}D_{\mf}} \left( \int_{2\kappa_X \lVert f \rVert_\infty^2/\rbar_{\mf}^2}^\infty e^{-nu/(2\kappa_X \lVert f \rVert_\infty^2)} \dd u + \int_0^{2\kappa_X \lVert f \rVert_\infty^2/\rbar_{\mf}^2} e^{-n\sqrt u /(2 \sqrt{\kappa_X} \rbar_{\mf} \lVert f \rVert_\infty)} \dd u \right)\\
  &\leq e^{-L_{\mf}D_{\mf}} \cdot \frac{2\kappa_X \lVert f \rVert_\infty^2}{n} \left( \int_0^\infty e^{-v} \dd v + \frac{2\rbar_{\mf}}{n} \int_0^\infty e^{-\sqrt v} \dd v \right)\\
  &\lesssim e^{-L_{\mf}D_{\mf}} \frac{C(C_{\rbar}, \lVert f \rVert_\infty)}{n} 
\end{align*}
which is the claim.
\end{proof}

\subsection{Subexponential terms}\label{subs:Laplace}

\begin{proposition}\label{prop:prob:Xi:xi}
	Let $\Xi_n^\xi(u) = \langle I_n^\xi - \Eb I_n^\xi, u \rangle$.
	For any symmetric function $u$,
	\begin{equation*}
	\Pb(\Xi_n^\xi(u) \geq t) \leq 2 \exp \left(  - \frac{1}{C} \min \left(  \frac{\pi^2 n t^2 \privpar^4}{64 \tau_n^4 \lVert u \rVert^2}, \frac{\sqrt{2\pi nt}\privpar}{4\tau_n \lVert u \rVert_\infty^{1/2}} \right) \right).
	\end{equation*}
\end{proposition}

\begin{proof}
	Let $H$ be the hyperplane orthogonal to the space generated by the vector $\vec 1$ in $\R^n$.
	Then, for $\xi = (\xi_1,\ldots,\xi_n)^\top$, $\xi - \xibar_n \vec 1 = P_H \xi$.
	We have
	\begin{align*}
	\Xi_n^\xi(u) &= \frac{1}{2\pi n} [(P_H \xi)^\top T_n(u) P_H \xi - \Eb (P_H \xi)^\top T_n(u) P_H \xi]. 
	\end{align*}
	
	We will now use Proposition~\ref{prop:goetze} from Appendix~\ref{app:aux} which is taken from \cite{goetze2019concentration}.
	More precisely, we would like to apply this result with our $\xi_i$ playing the role of the $X_i$, with $A = P_H^\top T_n(u) P_H$, and $\beta = 1$.
	We have $\Eb \xi_i^2  = \sigma_i^2 = 8\tau_n^2/\privpar^2$ for all $i \in \llbracket 1, n\rrbracket$.
	Moreover $\lVert \xi_i \rVert_{\Psi_1} \leq 4\tau_n / \privpar$ which will play the role of $M$.
	The last estimate is easily derived using the fact that $\lvert \xi_i \rvert$ obeys an exponential distribution with parameter $\lambda = \privpar/(2\tau_n)$ and then considering the moment generating function for the exponential distribution.
	It remains to bound the quantities $\lVert A \rVert_{\mathrm{HS}}$ and $\lVert A \rVert_{\mathrm{op}}$.
	First,
	\begin{align*}
	\lVert A \rVert_{\mathrm{HS}}^2 = \trace(A^\top A) &= \trace(P_H^\top T_n(u) P_H P_H^\top T_n(u) P_H)\\
&= \trace((P_H P_H^\top T_n(u))^2) \qquad [\text{cyclic property}]\\
	&= \rho(P_H P_H^\top)^2 \cdot \trace(T_n(u)^2) \qquad [\text{since } \trace((MN)^2) \leq \rho(M)^2  \trace(N^2)]\\
	&\leq \trace(T_n(u)^2).
	\end{align*}
	Using the same argument as on p.~284 in \cite{comte2001adaptive}, we have $\trace(T_n(u)^2) \leq n \lVert u \rVert^2$, and hence
	\begin{equation*}
	\lVert A \rVert_{\mathrm{HS}}^2 \leq n \lVert u \rVert^2.
	\end{equation*}
	Second, for $\lVert A \rVert_{\mathrm{op}}$ have the bound
	\begin{equation*}
	\lVert A \rVert_{\mathrm{op}} = \rho(A)  \leq \rho (T_n(u)) \leq \lVert u \rVert_\infty.
	\end{equation*}
	Thus, we finally obtain
	\begin{equation*}
	\Pb(\Xi_n^\xi(u) \geq t) \leq 2 \exp \left(  - \frac{1}{C} \min \left(  \frac{\pi^2 n t^2 \privpar^4}{64 \tau_n^4 \lVert u \rVert^2}, \frac{\sqrt{2\pi nt}\privpar}{4\tau_n \lVert u \rVert_\infty^{1/2}} \right) \right)
	\end{equation*}
	which is the claim assertion.
\end{proof}

\begin{lemma}\label{l:E:Gxi}
	For any fixed model $\mf \in \Mc_n$ and a sufficiently large constant $\kappa_\xi > 0$, we have
	\begin{equation*}
	\Eb \left[ \left( (G^\xi(\mf))^2 - \kappa_\xi \frac{\tau_n^4D_{\mf} (L_{\mf}^4 + L_\mf + \log(n) ) }{n\privpar^4}  \right)_+ \right] \lesssim e^{-L_{\mf}D_{\mf}} \frac{C(C_{\rbar}) \tau_n^4}{n^3 \privpar^4}.
	\end{equation*}
\end{lemma}

\begin{proof}
	As in the proof of Lemma~\ref{l:E:GX} we consider
\begin{align*}
\Pb(\sup_{u \in \Bc_\mf} \Xi^\xi_n(u) > \eta) &= \Pb \left[  \exists (u_k) \in \prod_{k \geq 0} T_k : \Xi_n^\xi(u_0) + \sum_{k \geq 1} \Xi_n^\xi(u_k - u_{k-1}) > \eta_0 + \sum_{k \geq 1} \eta_k \right]\\
&\leq P_1 + P_2
\end{align*}
with
\begin{align*}
P_1 &= \sum_{u_0 \in T_0} \Pb(\Xi^\xi_n(u_0) > \eta_0),\\
P_2 &= \sum_{k \geq 1} \sum_{\substack{u_{k-1} \in T_{k-1}\\u_k \in T_k}} \Pb ( \Xi_n^\xi(u_k - u_{k-1}) > \eta_k ).
\end{align*}
Now, for any $u_0 \in T_0$, we obtain from Proposition~\ref{prop:prob:Xi:xi} that
\begin{align*}
\Pb(\Xi^\xi_n(u_0) > \eta_0) &\leq 2 \exp \left( - \frac{1}{C}  \min \left( \frac{\pi^2 n\eta_0^2 \privpar^4}{64\tau_n^4 \lVert u_0 \rVert^2} , \frac{ \sqrt{2\pi n\eta_0} \privpar}{4\tau_n \lVert u_0 \rVert_\infty} \right)  \right)\\
&\leq 2 \exp \left( - \frac{1}{C}  \min \left( \frac{\pi^2 n\eta_0^2 \privpar^4}{64\tau_n^4 \delta_0^2} , \frac{\sqrt{2\pi n\eta_0} \privpar}{4\tau_n \sqrt{\rbar_\mf \delta_0}} \right)  \right)
\end{align*}
and hence
\begin{equation*}
P_1 \leq 2 \exp(H_0) \exp \left( - \frac{1}{C}  \min \left( \frac{\pi^2 n\eta_0^2 \privpar^4}{64\tau_n^4 \delta_0^2} , \frac{\sqrt{2\pi n\eta_0} \privpar}{4 \tau_n \sqrt{\rbar_{\mf}\delta_0}} \right)  \right)
\end{equation*}
We choose $\eta_0$ such that
\begin{equation*}
\min \left( \frac{\pi^2 n\eta_0^2 \privpar^4}{64\tau_n^4 \delta_0^2} , \frac{\sqrt{2\pi n\eta_0} \privpar}{4 \tau_n \sqrt{\rbar \delta_0}} \right) \geq H_0 + L_{\mf}D_{\mf} + \lambda
\end{equation*}
which in turn is satisfied whenever
\begin{equation*}
\eta_0 \geq C \cdot \frac{\tau_n^2 \delta_0}{\privpar^2} \max \left\lbrace \sqrt{\frac{H_0 + L_{\mf}D_{\mf} + \lambda}{n}}, \frac{\rbar_{\mf}}{n} (H_0 + L_{\mf}D_{\mf} + \lambda)^2 \right\rbrace 
\end{equation*}
for some sufficiently large constant $C>0$.
By Proposition~\ref{prop:prob:Xi:xi} for any choice of $u_{k-1}$ and $u_k$
\begin{align*}
\Pb ( \Xi_n^\xi(u_k - u_{k-1}) > \eta_k ) &\leq 2 \exp \left( - \frac{1}{C}  \min \left( \frac{\pi^2 n\eta_k^2 \privpar^4}{64\tau_n^4 \lVert u_k - u_{k-1} \rVert^2} , \frac{\sqrt{2\pi n\eta_k} \privpar}{4 \tau_n \lVert u_k - u_{k-1} \rVert_\infty^{1/2}} \right)  \right)\\
&\leq 2 \exp \left( - \frac{1}{C}  \min \left( \frac{\pi^2 n\eta_k^2 \privpar^4}{160\tau_n^4 \delta_{k-1}^2} , \frac{ \sqrt{2\pi n\eta_k} \privpar}{4\sqrt{3/2} \tau_n \sqrt{\rbar_{\mf} \delta_{k-1}}} \right)  \right).
\end{align*}
Thus,
\begin{equation*}
P_2 \leq 2 \sum_{k \geq 1} \exp(H_{k-1}) \exp(H_k) \exp \left( - \frac{1}{C}  \min \left( \frac{\pi^2 n\eta_k^2 \privpar^4}{160\tau_n^4 \delta_{k-1}^2} , \frac{\sqrt{2\pi n\eta_k} \privpar}{4\sqrt{3/2} \tau_n \sqrt{\rbar_{\mf} \delta_{k-1}}} \right)  \right).
\end{equation*}
Here we choose the $\eta_k$ such that
\begin{equation*}
\min \left( \frac{\pi^2 n\eta_k^2 \privpar^4}{160\tau_n^4 \delta_{k-1}^2} , \frac{\sqrt{2 \pi n\eta_k} \privpar}{4\sqrt{3/2} \tau_n \sqrt{\rbar_{\mf} \delta_{k-1}}} \right) \geq H_{k-1} + H_k + kD_{\mf} + L_{\mf} D_{\mf} + \lambda
\end{equation*}
which in turn is satisfied whenever
\begin{equation*}\scriptstyle
\eta_k \geq C \frac{\tau_n^2 \delta_{k-1}}{\privpar^2}\max \left\lbrace \sqrt{\frac{H_{k-1} + H_k + k D_{\mf} + L_{\mf} D_{\mf} + \lambda}{n}} , \frac{\rbar_{\mf} }{n} (H_{k-1} + H_k + k D_{\mf} + L_{\mf} D_{\mf} + \lambda)^2 \right\rbrace 
\end{equation*}
for some sufficiently large constant $C > 0$.
Under this choice of $(\eta_k)_{k \geq 0}$, we obtain for $\eta \geq \sum \eta_k$ (under the assumption that $D_{\mf} \geq 1$)
\begin{align*}
\Pb(\sup_{u \in \Bc_{\mf}} \Xi^\xi_n(u) > \eta) &\leq P_1 + P_2\\
&\leq 2 \exp(-L_{\mf}D_{\mf}- \lambda) + 2 \sum_k \exp(-kD_{\mf}- L_{\mf}D_{\mf} - \lambda)\\
&= 2 \exp (- L_{\mf}D_{\mf} - \lambda) \left[ 1 + \sum_k e^{-kD_{\mf}} \right]\\
&\leq 3.2 \exp (- L_{\mf}D_{\mf} - \lambda).
\end{align*}

Let us now find a bound for $\sum_{k \geq 0} \eta_k$.
We have
\begin{align*}\scriptstyle
\left( \sum_{k \geq 0} \eta_k \right)^2 &\scriptstyle \lesssim \left( \frac{ \tau_n^2 \delta_0}{\privpar^2} \left[ \sqrt{\frac{H_0 + L_{\mf}D_{\mf} + \lambda}{n}} + \frac{\rbar_{\mf}}{n} (H_0 + L_{\mf}D_{\mf} + \lambda)^2 \right]  \right. \\
&\scriptstyle \hspace{1em}+ \frac{\tau_n^2}{\privpar^2} \sum_k \delta_{k-1} \left[ \sqrt{\frac{H_{k-1} + H_k + k D_{\mf} + L_{\mf} D_{\mf} + \lambda}{n}}  + \frac{\rbar_{\mf} (H_{k-1} + H_k + k D_{\mf} + L_{\mf} D_{\mf} + \lambda)^2}{n}  \right]  \left. \right)^2\\
&\scriptstyle = \frac{\tau_n^4}{\privpar^4} \left[ \left(\delta_0 \sqrt{\frac{H_0 + L_{\mf}D_{\mf} + \lambda}{n}} + \sum_{k \geq 1} \delta_{k-1} \sqrt{\frac{H_{k-1} + H_k + k D_{\mf} + L_{\mf} D_{\mf} + \lambda}{n}}\right)   \right.\\
&\scriptstyle\hspace{1em} + \left.  \frac{\rbar_{\mf}}{n} \left( \delta_0 (H_0 + L_{\mf}D_{\mf} + \lambda)^2 + \sum_{k \geq 1} \delta_{k-1} (H_{k-1} + H_k + kD_{\mf} + L_{\mf} D_{\mf} + \lambda)^2 \right) \right]^2\\
&\scriptstyle \lesssim \frac{\tau_n^4}{\privpar^4} \left( \delta_0 \sqrt{\frac{H_0 + L_{\mf}D_{\mf} + \tau}{n}} + \sum_{k \geq 1} \delta_{k-1} \sqrt{\frac{H_{k-1} + H_k + k D_{\mf} + L_{\mf} D_{\mf} + \lambda}{n}} \right)^2\\
&\scriptstyle\hspace{1em} + \frac{\tau_n^4}{\privpar^4} \cdot \frac{\rbar_{\mf}^2}{n^2} \cdot \left( \delta_0 (H_0 + L_{\mf}D_{\mf} + \lambda)^2 + \sum_{k \geq 1} \delta_{k-1} (H_{k-1} + H_k + kD_{\mf} + L_{\mf} D_{\mf} + \lambda)^2 \right)^2\\
&\scriptstyle \lesssim \frac{\tau_n^4}{n \privpar^4} \left( \delta_0 + \sum_k \delta_{k-1} \right) \left( \delta_0 (H_0 + L_{\mf}D_{\mf} + \lambda) + \sum_{k \geq 1} \delta_{k-1} (H_{k-1} + H_k + kD_{\mf} + L_{\mf}D_{\mf} + \lambda)  \right) \\
&\scriptstyle\hspace{1em} + \frac{\tau_n^4 \rbar_{\mf}^2}{\privpar^4 n^2} \delta_0^2 \left[ H_0^2 + (L_{\mf}D_{\mf})^2 +  \lambda^2 + \sum 2^{-(k-1)} (H_{k-1}^2 + H_k^2 + k^2 D_{\mf}^2 + L_{\mf}^2 D_{\mf}^2 + \lambda^2) \right]^2\\
&\scriptstyle \lesssim \frac{\tau_n^4}{n \privpar^4} \cdot \delta_0^2 (H_0 + L_{\mf}D_{\mf} + \lambda + \sum_{k\geq 1} 2^{-(k-1)} (H_{k-1} + H_k + kD_{\mf} + L_{\mf}D_{\mf} + \lambda))\\
&\scriptstyle \hspace{1em}+ \frac{\tau_n^4 \rbar_{\mf}^2}{\privpar^4 n^2} \delta_0^2 \left[ D_{\mf}^2 \log^2(1/\delta_0) + L_{\mf}^2 D_{\mf}^2 + \lambda^2 + L_{\mf}^2 D_{\mf}^2 + \lambda^2 \right]^2 \\
&\scriptstyle \lesssim  \frac{\tau_n^4}{n \privpar^4} \cdot \delta_0^2 (H_0 + L_{\mf}D_{\mf} + \lambda +  D_{\mf} \log(1/\delta_0) +  L_{\mf}D_{\mf} + \lambda)\\
&\scriptstyle \hspace{1em}+ \frac{\tau_n^4 \rbar_{\mf}^2}{n^2 \privpar^4} \delta_0^2 \left[ L_{\mf}^2 D_{\mf}^2 + D_{\mf}^2 \log^2(1/\delta_0) + \lambda^2 \right]^2\\
&\scriptstyle \lesssim \frac{\tau_n^4}{n\privpar^4} \cdot \delta_0^2 \left[ L_{\mf} D_{\mf} + D_\mf \log(1/\delta_0)  + \lambda \right] \\
&\scriptstyle \hspace{1em}+  \frac{\tau_n^4 C_{\rbar}^2}{n \privpar^4} \delta_0^2 L_\mf^4 D_\mf^3 + \frac{\tau_n^4 C_{\rbar}^2}{n \privpar^4} \delta_0^2  \log^4(1/\delta_0) D_\mf^3 + \frac{\lambda^4 \tau_n^4 \rbar_\mf^2}{\privpar^4 n^2}\delta_0^2.
\end{align*}
Now, taking $\delta_0 = c/n$ for some numerical constant $0 < c < 1/5$, we obtain (note that we assume $D_\mf \leq n$ for all $\mf \in \Mc_n$)
\begin{align*}
	\left( \sum_{k \geq 1} \eta_k \right)^2 &\leq \kappa_\xi \left\lbrace  \frac{\tau_n^4}{n\privpar^4} (L_{\mf} + L_{\mf}^4 + \log^4(n) ) D_{\mf} +\frac{\tau_n^4}{n^3 \privpar^4} \left[ \lambda \vee \frac{\lambda^4 \rbar_{\mf}^2}{n} \right] \right\rbrace
\end{align*}
for a sufficiently large constant $\kappa_\xi=\kappa_\xi(C_{\rbar})$.
Finally,
\begin{align*}
\Eb &\left[ \left( (G^\xi(\mf))^2 - \kappa_\xi \frac{\tau_n^4D_{\mf} (L_{\mf}^4 + L_{\mf} +  \log^4(n) ) }{n\privpar^4}  \right)_+ \right]\\
&\leq \int_0^\infty \Pb \left( (G^\xi(\mf))^2 > \kappa_\xi \frac{\tau_n^4D_{\mf} (L_{\mf}^4 + L_{\mf} +  \log^4(n) ) }{n\privpar^4}  + u \right) \dd u\\
&\leq e^{-L_{\mf}D_{\mf}} \left( \int_{(n/\rbar_{\mf}^2)^{1/3}}^\infty e^{-n\privpar/(\tau_n \sqrt{\rbar_{\mf}}) \cdot (u/(2\kappa_\xi))^{1/4}} \dd u + \int_0^{(n/\rbar_{\mf}^2)^{1/3}} e^{-un^3\privpar^4/(2\kappa_\xi \tau_n^4)} \dd u \right)\\
&\leq e^{-L_{\mf}D_{\mf}} \cdot \left( \frac{2\kappa_\xi \tau_n^4\rbar_{\mf}^2}{n^4 \privpar^4} +  \frac{2\kappa_\xi \tau_n^4}{n^3\privpar^4} \right)\\
&\lesssim e^{-L_{\mf}D_{\mf}} \frac{\tau_n^4}{n^3 \privpar^4} 
\end{align*}
which finishes the proof.
\end{proof}

\subsection{Mixed terms}\label{subs:mixed}

\begin{proposition}\label{prop:prob:Xitilde}
	For any symmetric function $u$,
	\begin{equation*}
	\Pb \left(  \Xitilde_n(u) \geq t \right) \leq 2 \exp \left( - \frac{1}{C} \min \left( \frac{t^2 n}{2M^4 \lVert u \rVert^2 \cdot \lVert f\rVert_\infty} , \left( \frac{nt}{M^2 \lVert u\rVert_\infty \cdot \lVert f \rVert_\infty^{1/2}} \right)^{1/2}  \right) \right) 
	\end{equation*}
	where $M=3+4\tau_n/\privpar$.
\end{proposition}

\begin{proof}
In order to deal with the mixed term, we first write
\begin{equation*}
\begin{pmatrix}
X\\
\xi
\end{pmatrix} = 
\begin{pmatrix}
\sqrt{\Sigma_\Xb} & \boldsymbol 0_{n}\\
\boldsymbol 0_{n} & \Eb_{n} 
\end{pmatrix}
\begin{pmatrix}
Y\\
\xi
\end{pmatrix}
\end{equation*}
where $Y = (Y_1,\ldots,Y_n)^\top$ is a vector of i.i.d. standard Gaussian random variables.
Then, the term of interest can be written as
\begin{align*}
\begin{pmatrix}
X^\top & \xi^\top
\end{pmatrix}
\begin{pmatrix}
\boldsymbol 0_{n} & T_n(u)\\
T_n(u) & \boldsymbol 0_{n}
\end{pmatrix}
\begin{pmatrix}
X\\
\xi
\end{pmatrix} &= \begin{pmatrix}
Y^\top & \xi^\top
\end{pmatrix}
\begin{pmatrix}
\sqrt{\Sigma_\Xb} & \boldsymbol 0_{n}\\
\boldsymbol 0_{n} & \Eb_{n} 
\end{pmatrix}
\begin{pmatrix}
\boldsymbol 0_{n} & T_n(u)\\
T_n(u) & \boldsymbol 0_{n}
\end{pmatrix}
\begin{pmatrix}
\sqrt{\Sigma_\Xb} & \boldsymbol 0_{n}\\
\boldsymbol 0_{n} & \Eb_{n} 
\end{pmatrix}
\begin{pmatrix}
Y\\
\xi
\end{pmatrix}\\
&= \begin{pmatrix}
Y^\top & \xi^\top
\end{pmatrix} 
\begin{pmatrix}
\boldsymbol 0_{n} & \sqrt{\Sigma_\Xb} T_n(u) \\
T_n(u) \sqrt{\Sigma_\Xb} & \boldsymbol 0_{n}
\end{pmatrix}
\begin{pmatrix}
Y \\
\xi
\end{pmatrix} \\
&\eqdef \begin{pmatrix}
Y^\top & \xi
\end{pmatrix} A \begin{pmatrix}
Y\\
\xi
\end{pmatrix}.
\end{align*}
Since all components of the vector $(Y^\top \, \xi^\top)$ are independent, and the matrix $A$ is symmetric, we can apply  Proposition~\ref{prop:goetze} again with $\beta = 1$ as in the proof of Proposition~\ref{prop:prob:Xi:xi}.
We have $\Eb Y_i^2 = 1$, $\Eb \xi_i^2 = 8\tau_n^2/\privpar^2$.
As seen above $\lVert \xi_i \rVert_{\psi_1} \leq 4\tau_n/\privpar$ and moreover $\lVert Y_i \rVert_{\psi_1} \leq \lVert 1 \rVert_{\psi_2} \cdot \lVert Y_i \rVert_{\psi_2} \leq (\log 2)^{-1/2} \cdot \sqrt{3} \leq 3$.
Hence, we can take $M = 3 + 4\tau_n/\privpar$.
Application of Proposition~\cite{goetze2019concentration} yields
\begin{equation*}
\Pb \left(  \Xitilde_n(u) \geq t \right) \leq 2 \exp \left( - \frac{1}{C} \min \left( \frac{4\pi^2 t^2 n^2}{M^4 \lVert A \rVert_{\text{HS}}^2} , \left( \frac{2 \pi nt}{M^2 \lVert A \rVert_{\text{op}}} \right)^{1/2}  \right) \right),
\end{equation*}
and we have to find appropriate bounds for the quantities $\lVert A \rVert_{\mathrm{HS}}$ and $\lVert A \rVert_{\mathrm{op}}$.
Now, using the estimate H.1.g in Section II.9 from\cite{marshall2011inequalities}, p.~341, we have
\begin{align*}
\lVert A \rVert_{\text{HS}}^2 &= \trace(A^\top A) = \trace \begin{pmatrix}
\sqrt{\Sigma_\Xb} T_n(u)^2 \sqrt{\Sigma_\Xb} & \boldsymbol 0_{n}\\
\boldsymbol 0_{n} & T_n(u) \Sigma_\Xb T_n(u)
\end{pmatrix}\\
&= \trace (\sqrt{\Sigma_\Xb} T_n(u)^2 \sqrt{\Sigma_\Xb}) + \trace (T_n(u) \Sigma_\Xb T_n(u))\\
&= 2 \trace(\Sigma_X T_n(u)^2)\\
&\leq 2 n \lVert u \rVert^2 \cdot \lVert f \rVert_\infty.
\end{align*}
Finally, in order to bound $\lVert A \rVert_{\text{op}}$, note that
\begin{align*}
\lVert A \rVert_{\text{op}} &\leq \lVert T_n(u)\rVert_{\text{op}} \cdot \lVert \sqrt{\Sigma_X} \rVert_{\text{op}}\\
&\leq \lVert u\rVert_\infty \cdot \lVert f \rVert_\infty^{1/2}.
\end{align*}
\end{proof}

\begin{lemma}\label{l:E:Gtilde}
	For any fixed model $\mf \in \Mc_n$ and a sufficiently large constant $\kappatilde > 0$, we have
	\begin{align*}
	\Eb &\left[ \left( (\Gtilde(\mf))^2 - \kappatilde M^4 (1+\lVert f \rVert_\infty)^2 (L_{\mf}^4 + L_\mf + \log(n) ) \frac{D_{\mf}}{n} \right)_+ \right]\\
	&\hspace*{12em} \lesssim e^{-L_{\mf}D_{\mf}} \frac{C(C_{\rbar}, \lVert f \rVert_\infty) M^4}{n^3}
	\end{align*}
	where $M = 3+ 4\tau_n/\privpar$.
\end{lemma}

\begin{proof}
	We define $P_1$ and $P_2$ in analogy to the definition in the proof of Lemma~\ref{l:E:Gxi}, and using Proposition~\ref{prop:prob:Xitilde} we obtain
	\begin{align*}
	\Pb(\Xitilde_n(u_0) > \eta_0) &\leq 2 \exp \left( - \frac{1}{C}  \min \left( \frac{2\pi^2 n\eta_0^2}{\lVert f \rVert_\infty M^4 \lVert u_0 \rVert^2} , \frac{ \sqrt{2\pi n\eta_0}}{M \lVert f \rVert_\infty^{1/2}  \lVert u_0 \rVert_\infty} \right)  \right)\\
	&\leq 2 \exp \left( - \frac{1}{C}  \min \left( \frac{2\pi^2 n\eta_0^2}{\lVert f \rVert_\infty M^4 \delta_0^2} , \frac{\sqrt{2\pi n\eta_0}}{M \lVert f \rVert_\infty^{1/2} \sqrt{ \rbar_\mf \delta_0}} \right)  \right),
	\end{align*}
	and hence
	\begin{equation*}
	P_1 \leq 2 \exp(H_0) \exp \left( - \frac{1}{C}  \min \left( \frac{2 \pi^2 n\eta_0^2}{\lVert f \rVert_\infty M^4 \delta_0^2} , \frac{2 \pi \sqrt{n\eta_0}}{M \lVert f \rVert_\infty^{1/2} \sqrt{\rbar_\mf \delta_0}} \right) \right).
	\end{equation*}
	We choose $\eta_0$ such that
	\begin{equation*}
	 \frac{1}{C}  \min \left( \frac{n\eta_0^2}{2 \lVert f \rVert_\infty M^4 \delta_0^2} , \frac{\sqrt{n\eta_0}}{M \lVert f \rVert_\infty^{1/2} \sqrt{\rbar_\mf \delta_0}} \right) \geq H_0 + L_{\mf}D_{\mf} + \delta
	\end{equation*}
	which in turn is satisfied whenever
	\begin{equation*}
	\eta_0 \geq C M^2 \delta_0 (1 + \lVert f \rVert_\infty) \max \left\lbrace \sqrt{\frac{H_0 + L_{\mf}D_{\mf} + \lambda}{n}}, \frac{\rbar_{\mf} (H_0 + L_{\mf}D_{\mf} + \lambda)^2}{n} \right\rbrace
	\end{equation*}
	for some sufficiently large constant $C > 0$.
	From Proposition~\ref{prop:prob:Xitilde} we obtain for any choice of $u_{k-1}$ and $u_k$ that
	\begin{align*}
	\Pb ( \Xitilde_n (u_k - u_{k-1}) > \eta_k ) &\leq 2 \exp \left( - \frac{1}{C}  \min \left( \frac{2 \pi^2n\eta_k^2}{\lVert f \rVert_\infty M^4  \lVert u_k - u_{k-1} \rVert^2} , \frac{\sqrt{2\pi n\eta_k} }{M \lVert f \rVert_\infty^{1/2} \lVert u_k - u_{k-1} \rVert_\infty^{1/2}} \right)  \right)\\
	&\leq 2 \exp \left( - \frac{1}{C}  \min \left( \frac{4\pi^2 n\eta_k^2}{5 \lVert f \rVert_\infty M^4 \delta_{k-1}^2} , \frac{2\sqrt{\pi n\eta_k} }{M \lVert f \rVert_\infty^{1/2} \sqrt{3\rbar_{\mf} \delta_{k-1}}} \right)  \right).
	\end{align*}
	As a consequence,
	\begin{equation*}
	P_2 \leq 2 \sum_{k \geq 1} \exp(H_{k-1}) \exp(H_k) \exp \left( - \frac{1}{C}  \min \left( \frac{n\eta_k^2}{5 \lVert f \rVert_\infty M^4 \delta_{k-1}^2} , \frac{2\sqrt{\pi n\eta_k} }{M \lVert f \rVert_\infty^{1/2} \sqrt{3\rbar_{\mf} \delta_{k-1}}} \right) \right).
	\end{equation*}
	Here we choose the $\eta_k$ such that
	\begin{equation*}
	\frac{1}{C}\min \left( \frac{4 \pi^2 n\eta_k^2}{5 \lVert f \rVert_\infty M^4 \delta_{k-1}^2} , \frac{2 \sqrt{\pi n\eta_k} }{M \lVert f \rVert_\infty^{1/2} \sqrt{3\rbar_{\mf} \delta_{k-1}}} \right) \geq H_{k-1} + H_k + kD_{\mf} + L_{\mf} D_{\mf} + \lambda,
	\end{equation*}
	which in turn is satisfied whenever
	\begin{equation*}\scriptstyle
	\eta_k \geq C M^2 \delta_{k-1} (1+ \lVert f \rVert_\infty) \max \left\lbrace \sqrt{\frac{H_{k-1} + H_k + kD_{\mf} + L_{\mf} D_{\mf} + \lambda}{n}}, \frac{\rbar_{\mf}(H_{k-1} + H_k + kD_{\mf} + L_{\mf} D_{\mf} + \lambda)^2}{n} \right\rbrace.
	\end{equation*}
	Apart from the dependence of the leading numerical constant on $\lVert f\rVert_\infty$ and the different dependence in terms of $\privpar$ (which is hidden in the quantity $M$), the obtained expressions for $\eta_k$, $k\geq 0$ are the same as in the proof of Lemma~\ref{l:E:Gxi}.
	Taking $\delta_0 = c/n$ for some numerical constant $0 < c < 1/5$ again, we obtain
	\begin{align*}
		\left( \sum_{k \geq 1} \eta_k \right)^2 &\leq \kappatilde M^4 (1+ \lVert f \rVert_\infty)^2  \left\lbrace \frac{D_{\mf}}{n} ( L_{\mf} + L_{\mf}^4 + \log(n) )  +  + \frac{1}{n^3} \left[ \lambda \vee \frac{\lambda^4 \rbar_{\mf}^2}{n} \right] \right\rbrace.
	\end{align*}
	A calculation similar to the one in the proof of Lemma~\ref{l:E:Gxi} yields
	\begin{align*}
		\Eb &\left[ \left( (\Gtilde(\mf))^2 - \kappatilde M^4 (1+\lVert f \rVert_\infty)^2 (L_\mf + L_{\mf}^4 + \log(n) ) \frac{D_{\mf}}{n}  \right)_+ \right]\\
&\hspace{20em}\leq e^{-L_{\mf}D_{\mf}} \frac{C(C_{\rbar}, \lVert f \rVert_\infty) M^4}{n^3}.
	\end{align*}
\end{proof} 
\section{Auxiliary results}\label{app:aux}

\begin{proposition}[Hanson-Wright inequality, \cite{rudelson2013hanson}, Theorem~1.1]\label{HWI}
	Let $X=(X_1,\ldots,X_n) \in \R^n$ be a random vector with independent components $X_i$ which satisfy $\Eb X_i = 0$ and $\lVert X_i \rVert_{\psi_2} \leq K$.
	Let $A$ be an $n \times n$-matrix.
	Then, for every $t \geq 0$,
	\begin{equation*}
	\Pb \left( \lvert X^\top A X - \Eb X^\top A X \rvert > t \right) \leq 2 \exp \left[ - c \min \left( \frac{t^2}{K^4 \lVert A \rVert^2_{\mathrm{HS}}}, \frac{t}{K^2 \lVert A \rVert_{\mathrm{op}}} \right) \right].
	\end{equation*}
\end{proposition}
The following result generalizes Proposition~\ref{HWI} because it can also deal with other exponential Orlicz norms than $\lVert \cdot \rVert_{\psi_2}$.
This permits to apply the result to subexponential random variables as the Laplace noise used for our anonymization algorithm.
\begin{proposition}[\cite{goetze2019concentration}, Proposition~1.1]\label{prop:goetze}
	Let $X_1,\ldots,X_n$ be independent random variables satisfying $\Eb X_i = 0$, $\Eb X_i^2 = \sigma_i^2$, $\lVert X_i \rVert_{\psi_\beta} \leq M$ for some $\beta \in (0,1] \cup \{2\}$, and $A$ be a symmetric $n \times n$ matrix.
	For any $t > 0$ we have
	\begin{equation*}\Pb \left( \lvert \sum_ {i,j} a_{ij}X_i X_j - \sum_{i=1}^{n} \sigma_i^2 a_{ii} \rvert \geq t \right) \leq 2 \exp \left( - \frac{1}{C} \min \left( \frac{t^2}{M^4 \lVert A \rVert^2_{\mathrm{HS}}} , \left( \frac{t}{M^2 \lVert A \rVert_{\mathrm{op}}} \right)^{\beta/2} \right) \right).
	\end{equation*}
\end{proposition}
 
\bigskip

\section*{Acknowledgements}

This research has been supported in part by the research grant DFG DE 502/27-1 of the German Research Foundation (DFG). 
\printbibliography

\end{document}